\documentclass[a4paper,10pt]{amsart}
\usepackage[utf8]{inputenc}
\usepackage[T2A,T1]{fontenc}

\usepackage{enumerate}

\usepackage{mathrsfs}
\usepackage{amsfonts}
\usepackage{amsmath}
\usepackage{amssymb}
\usepackage{amsthm}
\usepackage{stmaryrd}
\usepackage{textcomp}
\usepackage{mathtext}
\usepackage{yfonts}

\usepackage{amscd}
\usepackage{tikz}
\usepackage{tikz-cd}
\usepackage{sidecap}
\usetikzlibrary{matrix,arrows}

\usepackage{mwe} \usepackage{subcaption}

\usepackage{url}

\usepackage{float}
\floatstyle{boxed}
\restylefloat{table}
\restylefloat{figure}

\newtheoremstyle{martin}{9pt}{9pt}{}{}{\bfseries}{.}{ }{}
\theoremstyle{martin}

\newtheorem{theorem}{Theorem}[section]
\newtheorem{lemma}[theorem]{Lemma}

\newtheorem{remark}[theorem]{Remark}
\newtheorem{example}[theorem]{Example}

\newcommand{\sectionintrolinebreak}{}

\makeatletter
\newcommand\suchthat{\@ifstar
  {\mathrel{}\middle|\mathrel{}}
  {\mid}}
\makeatother

\usepackage{nicefrac}

\newcommand{\vol}{\operatorname{vol}}

\newcommand{\nullspace}{{\{0\}}}
\newcommand{\inv}{{-1}}

\newcommand{\grad}{\operatorname{grad}}

\newcommand{\curl}{\operatorname{curl}}

\newcommand{\divergence}{\operatorname{div}}

\newcommand{\DC}{{\operatorname{DC}}}

\newcommand{\trace}{\operatorname{tr}}

\newcommand{\Ext}{\operatorname{Ext}}

\newcommand{\ext}{\operatorname{ext}}

\newcommand{\range}{\operatorname{ran}}

\newcommand{\cartan}{{\mathsf d}}

\newcommand{\set}[1]{ { \left\{ {#1} \right\} } }
\newcommand{\textif}{{\text{ if }}}

\newcommand{\Ned}{{\bfN\bf{d}}}
\newcommand{\RT}{{\bfR\bfT}}
\newcommand{\BDM}{{\bfB\bfD\bfM}}

\newcommand{\linhull}{{\operatorname{span}}}

\newcommand{\bbN}{{\mathbb N}}

\newcommand{\bbR}{{\mathbb R}}

\newcommand{\bbZ}{{\mathbb Z}}

\newcommand{\bfB}{{\mathbf B}}

\newcommand{\bfD}{{\mathbf D}}

\newcommand{\bfH}{{\mathbf H}}

\newcommand{\bfL}{{\mathbf L}}
\newcommand{\bfM}{{\mathbf M}}
\newcommand{\bfN}{{\mathbf N}}

\newcommand{\bfR}{{\mathbf R}}

\newcommand{\bfT}{{\mathbf T}}

\newcommand{\calC}{{\mathcal C}}

\newcommand{\calP}{{\mathcal P}}

\newcommand{\calS}{{\mathcal S}}
\newcommand{\calT}{{\mathcal T}}
\newcommand{\calU}{{\mathcal U}}

\newcommand{\calW}{{\mathcal W}}

\newcommand{\sfP}{{\mathsf P}}

\newcommand{\scrA}{{\mathscr A}}

\newcommand{\scrS}{{\mathscr S}}

\newcommand{\vecn}{{\vec n}}

\begin{document}

\title
[Higher Order FEEC]
{Higher-order finite element de Rham complexes,\\
partially localized flux reconstructions,\\
and applications}

\author{Martin Werner Licht}

\keywords{finite element exterior calculus, equilibrated a~posteriori error estimation, flux reconstruction, hp-adaptive finite element method}

\subjclass[2000]{65N30, 58A12}

\begin{abstract}
 We construct finite element de~Rham complexes of higher and possibly non-uniform polynomial order in finite element exterior calculus (FEEC). 
Starting from the finite element differential complex of lowest-order, known as the complex of Whitney forms,
we incrementally construct the higher-order complexes by adjoining exact local complexes associated to simplices.
We define a commuting canonical interpolant. 
On the one hand, this research provides a base for studying $hp$-adaptive methods in finite element exterior calculus.
On the other hand, our construction of higher-order spaces enables a new tool in numerical analysis which we call ``partially localized flux reconstruction''.
One major application of this concept is in the area of equilibrated a~posteriori error estimators: 
we generalize the Braess-Sch\"oberl error estimator to edge elements of higher and possibly non-uniform order. 
 \end{abstract}

\maketitle

\section{Introduction} \label{sec:intro}

The formalism of differential complexes offers a theoretical access to many partial differential equations in physics and engineering.
The numerical analysis of partial differential equations has embraced finite element differential complexes in understanding mixed finite element methods, such as for computational electromagnetism~\cite{hiptmair2001higher,hiptmair2002finite,Monk2003,AFW1}. 
Whereas many theoretical contributions utilize classical vector calculus, 
the calculus of differential forms enables a unified treatment of differential operators such as the gradient, the curl, or the divergence.
The wide adoption of the calculus of differential forms in the theory of partial differential equations 
has driven the study of finite element differential forms in numerical analysis~\cite{bossavit1988mixed,hiptmair2002finite}.
This line of thought has culminated in \emph{finite element exterior calculus} (FEEC,~\cite{AFW1,AFW2}),
which is the mathematical formalism that we adopt in this contribution.

Research efforts in finite element exterior calculus have focused on spaces of uniform polynomial order. 
Finite element spaces with spatially varying non-uniform polynomial order, however, are constitutive for $p$-adaptive and $hp$-adaptive finite element methods in numerical electromagnetism~\cite{S98_283,demkowicz2006computing,rachowicz2006fully}. 
We recall that $h$-adaptive methods refine the mesh locally but keep the polynomial order fixed,
that $p$-adaptive methods keep the mesh fixed but locally increase the polynomial order,  
and that $hp$-adaptive methods combine local mesh refinement and variation of the polynomial order. 
The latter form of adaptivity enables the efficient approximation of scalar functions with spatially varying smoothness or isolated singularities,
for example by Lagrange elements with non-uniform polynomial order. 
The theory of $hp$-adaptive mixed finite element methods in numerical electromagnetism requires differential complexes of spaces of non-uniform polynomial order, 
which include generalizations of N\'ed\'elec elements and Raviart-Thomas elements~\cite{MONK1994117,Ainsworth20016709,Demkowicz2001,schoberl2005high}.

In this contribution, we study the algebraic and structural properties of finite element de~Rham complexes of higher polynomial order.
We develop a formalism for finite element spaces of non-uniform polynomial order via finite element exterior calculus and construct a commuting interpolant.
This prepares future research on $hp$-adaptive methods in FEEC.
The main contribution of our research effort is a \emph{partially localized flux reconstruction},
which has an application to the efficient implementation of equilibrated a~posteriori error estimators: 
we generalize the locally constructed Braess-Sch\"oberl a~posteriori error estimator for edge elements~\cite{BrSchoMax} to the higher-order case.

It is common practice in the literature on $hp$-FEM to characterize approximation spaces 
by assigning a polynomial order $r_S \in \bbN_0$ to each simplex $S$ of the mesh,
such that the following \emph{hierarchy condition} is satisfied:
simplices have an associated polynomial order at least as large as the ones associated to their subsimplices. 
For example, this leads to spaces of scalar functions whose restriction to each simplex $S$ is a polynomial of order at most $r_S$.
We extend this concept to finite element exterior calculus: 
on each simplex we fix not only the polynomial order but we also choose between the $\calP_r$-family and $\calP_r^{-}$-family of finite element spaces.
For example, this allows one to choose between associating Raviart-Thomas spaces or Brezzi-Douglas-Marini spaces to a triangle.
The choice of spaces is subject to a similar hierarchy condition as in the scalar case.

We build on the intuition that finite element de~Rham complexes of higher (uniform or non-uniform) polynomial order
are constructed from the lowest-order finite element de~Rham complex by \emph{local augmentations} with local higher-order complexes.
A variant of this idea was already used in~\cite{schoberl2005high} and~\cite{zaglmayr2006high}. 
We may recall that the lowest-order finite element de~Rham complex is precisely the differential complex of Whitney forms~\cite{whitney2012geometric,bossavit1988mixed,hiptmair2002finite}.
In three dimensions, the latter translates into the well-known differential complex 
\begin{gather}
 \label{intro:lowestordercomplex}
 \begin{CD}
  \calP_1(\calT) @>{\grad}>> \Ned_{0}(\calT) @>{\curl}>> \RT_{0}(\calT) @>{\divergence}>> \calP_{0,\DC}(\calT)
 \end{CD}
\end{gather}
with respect to a triangulation $\calT$ of a three-dimensional domain~\cite{bossavit1988mixed,AFW1}. 
Here, $\calP_1(\calT)$ denotes piecewise affine Lagrange elements, 
$\Ned_{0}(\calT)$ denotes lowest-order N\'ed\'elec elements,
$\RT_{0}(\calT)$ denotes lowest-order Raviart-Thomas elements,
and $\calP_{0,\DC}(\calT)$ is spanned by the piecewise constant functions.

For a simple example of how to augment this complex 
with a local higher-order sequence, 
we fix $r \in \bbN$ and a tetrahedron $T \in \calT$ and consider the differential complex 
\begin{gather}
 \label{intro:localcomplex}
 \begin{CD}
  \mathring\calP_{r+1}(T)
  @>{\grad}>>
  \mathring\Ned_{r}(T)
  @>{\curl}>>
  \mathring\RT_{r}(\calT)
  @>{\divergence}>>
  \calP_{r}(T) \cap \ker \int_T.
 \end{CD}
\end{gather}
The first three spaces are the higher-order Lagrange space, the N\'ed\'elec space, and the Raviart-Thomas space over $T$
with Dirichlet, tangential, and normal boundary conditions, respectively, along $\partial T$. 
The final space is the order $r$ polynomials over $T$ with vanishing mean value. 
This sequence is exact and supported only over $T$,
and we augment the lowest-order complex~\eqref{intro:lowestordercomplex} by taking the direct sum with~\eqref{intro:localcomplex}.  

Similarly, we may first associate an exact finite element sequence to any lower dimensional subsimplex
and then extend the spaces onto the local patch around the subsimplex.   
If the extension operators commute with the exterior derivative, then the extended spaces over the local patch constitute a differential complex.
Moreover, if the choice of finite element spaces reflects the inclusion ordering of simplices, which we call \emph{hierarchy condition},
then the global finite element spaces constitute a differential complex. 

The degrees of freedom of the global higher-order spaces 
are the direct sum of the degrees of freedom for the Whitney forms 
and the degrees of freedom of the local higher-order spaces.
Following~\cite{deRhamHPFEM}, 
we express the degrees of freedom via Hodge decompositions and define the commuting canonical interpolant onto the finite element de~Rham complex.

Expressing finite element spaces by augmenting the lowest-order finite element space 
allows us to formalize a new concept in finite element methods
that we call \emph{partially localized flux reconstructions}.
In this context, \emph{flux reconstruction} refers to computing a generalized inverse of the exterior derivative between finite element spaces of differential forms. 
An example in the language of vector calculus is computing a generalized inverse of the mapping $\curl : \Ned_{r}(\calT) \rightarrow \RT_{r}(\calT)$
from order $r$ N\'ed\'elec elements to order $r$ Raviart-Thomas elements.

Algorithmically, we can tackle the problem either with a mixed finite element method or by solving normal equations,
either of which require solving a global linear problem over a higher-order finite element space. 
Our flux reconstruction reduces the global problem to the lowest-order case.
For example, assume that $\omega \in \RT_{r}(\calT)$ is the curl of a member of $\Ned_{r}(\calT)$.
We show how to decompose $\omega = \omega_0 + \curl \xi_r$, where $\omega_0 \in \RT_{0}(\calT)$ 
is the canonical interpolation of $\omega$ onto the lowest-order Raviart-Thomas space,
and where $\xi_r \in \Ned_{r}(\calT)$ is constructed by solving independent local linear problems.
These local problems are associated to single tetrahedra, 
and their size and well-posedness depend only on the local polynomial order and mesh quality;
they are independent of each other and hence accessible to parallelization.
It can be shown that there exists $\xi_0 \in \Ned_{0}(\calT)$ with $\curl \xi_0 = \omega_0$,
and so $\xi := \xi_0 + \xi_r \in \Ned_{r}(\calT)$ satisfies $\curl \xi = \omega$. 
We compute the vector field $\xi_0$ by solving a global problem only on a smaller lowest-order space.
The flux reconstruction is \emph{partially localized} in the sense that only the lowest-order terms require a global computation. 

One application of the partially localized flux reconstruction is determining the cohomology groups of finite element de~Rham complex with varying polynomial order. 
Specifically, it facilitates an easy proof that the canonical interpolant onto the Whitney forms induces isomorphisms on cohomology. 
Though we will not focus on that too much,
it also enables estimates for the constants in Poincar\'e-Friedrichs constants in finite element de~Rham complexes. 
We remark that a flux reconstruction for discrete distributional differential forms has enabled estimates for higher-order finite element Poincar\'e-Friedrichs constants in terms of such constants for the complex of Whitney forms~\cite{christiansen2020poincare}.

A major application is solving an open problem in the theory of \emph{equilibrated} a~posteriori error estimators.
They are based on the Prager-Synge identity. 
Those estimators have attracted persistent research efforts because they provide reliable and constant-free upper bounds for error of finite element methods 
\cite{repin2009two,ern2010guaranteed,ainsworth2011posteriori,repin2008posteriori}.
Efficient algorithms for finite element flux reconstruction are critical to make the estimator competitive in computations~\cite{Braess2007,ern2010posteriori,becker2015robust,becker2016local}.
In the case of the Poisson problem, a fully localized flux reconstruction for the divergence operator $\divergence: \RT_{r}(\calT) \rightarrow \calP_{r,\DC}(\calT)$
is possible if the Galerkin solution for the Poisson problem is given as additional information;
the resulting estimator is competitive~\cite{Braess2007,braess2009equilibrated,carstensen2010estimator}. 
Much less is known for equilibrated error estimators in numerical electromagnetism. 
Braess and Sch\"oberl have introduced an equilibrated a~posteriori error estimator 
for the $\curl$-$\curl$ problem over lowest-order N\'ed\'elec elements~\cite{BrSchoMax}. 
Analogously to the Poisson problem, they provide a fully localized flux reconstruction for the curl operator
$\curl: \Ned_{0}(\calT) \rightarrow \RT_{0}(\calT)$ which uses a Galerkin solution of the $\curl$-$\curl$-problem as additional information to achieve full localization. 
Extending this result to higher-order edge elements has remained an open problem.
But our partially localized flux reconstruction enables a \emph{fully localized flux reconstruction} even in the higher-order case.
In effect, we generalize equilibrated a~posteriori error estimators for the $\curl$-$\curl$ problem to the case of edge elements of higher and possibly non-uniform polynomial order. 

There are several directions for future research. 
Whereas the flux reconstruction has been constructed within the framework of finite element exterior calculus, 
we can only apply it to a~posteriori error estimation for edge elements in two and three dimensions. 
It is of not only theoretical but also practical interest to generalize the lowest-order flux reconstruction of~\cite{BrSchoMax} to general Whitney $k$-forms.
Beyond that focus on error estimation, 
it is of further interest in how far the techniques of finite element exterior calculus can contribute to bases with low condition numbers, improved sparsity properties, or fast algorithms~\cite{melenk2002condition,beuchler2007sparse,kirby2012fast,kirby2014low} for vector-valued finite element spaces.

The remainder of this article is structured as follows. 
In Section~\ref{sec:diffforms}, we recapitulate smooth and polynomial differential forms. 
In Section~\ref{sec:simplexcomplexes}, we consider exact sequences of polynomial differential forms over simplices. 
In Section~\ref{sec:whitneycomplex}, we study the complex of Whitney forms over a triangulation and the associated commuting interpolator.
Section~\ref{sec:hierarchycomplexes} discusses finite element de~Rham complexes of higher order.
In Section~\ref{sec:fluxreconstruction}, the partially localized flux reconstruction is introduced. 
Section~\ref{sec:application} eventually demonstrates the application to the Braess-Sch\"oberl error estimator.

\section{Smooth and Polynomial Differential Forms} \label{sec:diffforms}

In this section, we briefly recapitulate the calculus of differential forms on simplices.
We subsequently give a summary of the $\calP_r$ and $\calP_r^{-}$ families of spaces of polynomial differential forms. 
We point out Agricola and Friedrich's monograph~\cite{FriAgri} as a general reference on differential forms. 
Our discussion of polynomial differential forms is based on Arnold, Falk and Winther's seminal publication~\cite{AFW1}.
We only give a small outline and refer the reader to these sources for a thorough treatment. 

We agree on some notation. 
For $a \in \{0,1\}$, $k \in \bbZ$, and $n \in \bbN_0$ we let $\Sigma(a:k,0:n)$ denote the set of strictly increasing mappings
from $\{ a, \dots, k \}$ to $\{ 0, \dots, n \}$.
Note that $\Sigma(a:k,0:n) = \emptyset$ if $k > n$
and that $\Sigma(a:k,0:n) = \{\emptyset\}$ if $k < a$. 
For $n \in \bbN_0$ we let $A(n)$ denote the set of \emph{multiindices} in $n+1$ variables,
i.e.\ the set of functions from $\{0,\dots,n\}$ to $\bbN_0$.
The absolute value of $\alpha \in A(n)$ is $|\alpha| := \alpha(0) + \dots + \alpha(n)$. 
For $r \in \bbN$ we let $A(r,n)$ be the set of multiindices with absolute value $r$.
\\

We let $T \subset \bbR^{N}$ be a fixed but arbitrary $n$-dimensional simplex. 
We henceforth write $\dim T$ for the dimension of any simplex,
which one less than the number of its vertices. 
We write 
$\vol(T)$ for the $n$-dimensional volume of $T$.
We write $\Delta(T)$ for the set of subsimplices of $T$.
If $F \in \Delta(T)$, then $\imath_{F,T} : F \rightarrow T$ denotes the inclusion
in the sense of manifolds with corners.
Note that for $F \in \Delta(T)$ and $f \in \Delta(F)$ we have $\imath_{f,T} = \imath_{F,T} \imath_{f,F}$. 
Throughout this article, we assume that each simplex is equipped with an arbitrary but fixed orientation. 
\\

We let $C^{\infty}(T)$ denote the space of smooth functions over $T$
that are restrictions of a smooth function over $\bbR^{n}$.
More generally, for $k \in \bbZ$ we let $C^{\infty}\Lambda^{k}(T)$ denote the space of smooth differential $k$-forms over $T$
that are the pullback of a smooth differential $k$-form over $\bbR^{n}$ along the embedding of the simplex. 
We have $C^{\infty}\Lambda^{0}(T) = C^{\infty}(T)$ and $C^{\infty}\Lambda^{k}(T) = \emptyset$ for $k \notin \{ 0, \dots, n \}$.
When $\omega \in C^{\infty}\Lambda^{k}(T)$ and $\eta \in C^{\infty}\Lambda^{l}(T)$, 
then $\omega \wedge \eta \in C^{\infty}\Lambda^{k+l}(T)$ denotes the exterior product.
We recall that $\omega \wedge \eta = (-1)^{kl} \eta \wedge \omega$. 
We also remember the exterior derivative 
\begin{gather*}
 \cartan^{k} : C^{\infty}\Lambda^{k}(T) \rightarrow C^{\infty}\Lambda^{k+1}(T), 
\end{gather*}
which satisfies the differential property $\cartan^{k+1} \cartan^{k} \omega = 0$ for all $\omega \in C^{\infty}\Lambda^{k}(T)$.

If $T' \subset \bbR^{N}$ is another simplex and $\Phi : T' \rightarrow T$ is a smooth embedding,
then we have a pullback mapping $\Phi^{\ast} : C^{\infty}\Lambda^{k}(T) \rightarrow C^{\infty}\Lambda^{k}(T')$ for each $k \in \bbZ$.
One can show that $\Phi^{\ast} \cartan^{k} \omega = \cartan^{k} \Phi^{\ast} \omega$ for each $\omega \in C^{\infty}\Lambda^{k}(T)$.

The trace operator $\trace_{T,F}^{k} : C^{\infty}\Lambda^{k}(T) \rightarrow C^{\infty}\Lambda^{k}(F)$
is defined as the pullback of $k$-forms along $\imath_{F,T}$.
We have $\trace_{F,f}^{k} \trace_{T,F}^{k} = \trace_{T,f}^{k}$ for $F \in \Delta(T)$ and $f \in \Delta(F)$.
Moreover, these traces are surjective. 
\\

We write $\{ v_0^{T}, \dots, v_n^{T} \}$ for the set of vertices of $T$.
The barycentric coordinates $\{ \lambda_0^{T}, \dots, \lambda_{n}^{T} \}$ are the unique affine functions over $T$ 
that satisfy $\lambda_{i}^{T}(v_j^{T}) = \delta_{ij}$ for $0 \leq i,j \leq n$. 
We introduce the barycentric monomials $\lambda^{\alpha}_{T} := \prod_{i=0}^{n} \left( \lambda_i^{T} \right)^{\alpha(i)}$ for $\alpha \in A(n)$,
and define the space $\calP_{r}(T)$ of barycentric polynomials up to order $r$ 
as the span of barycentric monomials $\{ \lambda^{\alpha}_{T} \}_{ \alpha \in A(r,n) }$ of order $r$.

We define the barycentric $k$-alternators as the differential $k$-forms
\begin{gather*}
 \cartan\lambda_{\sigma}^{T} := \cartan\lambda_{\sigma(0)}^{T} \wedge \dots \wedge \cartan\lambda_{\sigma(k)}^{T}, 
 \quad 
 \sigma \in \Sigma(1:k,0:n), 
\end{gather*}
and the barycentric Whitney $k$-forms as 
\begin{gather*}
 \phi_{\rho}^{T} := \sum_{i=0}^{k} (-1) \lambda_{\rho(i)}^{T} 
 \cartan\lambda_{\rho - \rho(i)}^{T},
 \quad 
 \rho \in \Sigma(0:k,0:n).  
\end{gather*}
We then define 
\begin{align}
 \label{math:space:pr}
 \calP_r\Lambda^{k}(T) 
 &:=
 \linhull\left\{ \; 
  \lambda^{\alpha}_{T} \cartan\lambda_{\sigma}^{T} \;\suchthat*\; \alpha \in A(r,n), \; \sigma \in \Sigma(1:k,0:n) 
 \; \right\},
 \\
 \label{math:space:prminus}
 \calP_r^{-}\Lambda^{k}(T) 
 &:=
 \linhull\left\{ \; 
  \lambda^{\alpha}_{T} \phi_{\rho}^{T} \;\suchthat*\; \alpha \in A(r-1,n), \; \rho \in \Sigma(0:k,0:n) 
 \; \right\}.
\end{align}
We adhere to the convention that $\calP_r\Lambda^{k}(T) = \nullspace$ and $\calP_r^{-}\Lambda^{k}(T) = \nullspace$ for negative polynomial order $r$.
Note that $\calP_r(T) = \calP_{r}\Lambda^{0}(T)$. 

If $F \in \Delta(T)$ is a subsimplex, then 
\begin{align}
 \label{math:polydiffform:simplex:trace}
 \calP_{r}\Lambda^{k}(F)
 =
 \trace_{T,F}^{k} \calP_{r}\Lambda^{k}(T)
 ,\quad 
 \calP_{r}^{-}\Lambda^{k}(F)
 =
 \trace_{T,F}^{k} \calP_{r}^{-}\Lambda^{k}(T)
 .
\end{align}
Via the traces, we define spaces with boundary conditions. We write 
\begin{gather}
 \label{math:space:pr:bc}
 \mathring\calP_{r}\Lambda^{k}(T) 
 :=
 \left\{ \;
  \omega \in \calP_{r}\Lambda^{k}(T) 
  \;\suchthat*\; 
  \forall F \in \Delta(T) : \trace_{T,F}^{k} \omega = 0
 \;\right\}
 ,\\
 \label{math:space:prminus:bc}
 \mathring\calP_{r}^{-}\Lambda^{k}(T) 
 :=
 \left\{ \;
  \omega \in \calP_{r}^{-}\Lambda^{k}(T) 
  \;\suchthat*\; 
  \forall F \in \Delta(T) : \trace_{T,F}^{k} \omega = 0
 \;\right\}
 .
\end{gather}
These spaces are affinely invariant in the following sense. 
For every bijective affine mapping $\Phi : T' \rightarrow T$ from a simplex $T'$ onto $T$ we have 
\begin{align}
 \label{math:polydiffform:simplex:affineinvariance}
 \calP_{r}\Lambda^{k}(T')
 =
 \phi^{\ast}\calP_{r}\Lambda^{k}(T)
 ,\quad 
 \calP_{r}^{-}\Lambda^{k}(T')
 =
 \phi^{\ast}\calP_{r}^{-}\Lambda^{k}(T)
 , 
 \\
 \label{math:polydiffform:simplex:affineinvariance:bc}
 \mathring\calP_{r}\Lambda^{k}(T')
 =
 \phi^{\ast} \mathring\calP_{r}\Lambda^{k}(T)
 ,\quad 
 \mathring\calP_{r}^{-}\Lambda^{k}(T')
 =
 \phi^{\ast} \mathring\calP_{r}^{-}\Lambda^{k}(T)
 .
\end{align}
We recall some further inclusions and identities. One can show that  
\begin{gather}
 \label{math:simplex:polydiffform:inclusion}
 \calP_{r}\Lambda^{k}(T)
 \subseteq \calP_{r+1}^{-}\Lambda^{k}(T)
 \subseteq \calP_{r+1}\Lambda^{k}(T),
 \\
 \label{math:simplex:polydiffform:inclusion:bc}
 \mathring\calP_{r}\Lambda^{k}(T)
 \subseteq \mathring\calP_{r+1}^{-}\Lambda^{k}(T)
 \subseteq \mathring\calP_{r+1}\Lambda^{k}(T),
 \\
 \label{math:simplex:polydiffform:diffinclusion}
 \cartan^{k} \calP_{r}\Lambda^{k}(T) \subseteq \calP_{r-1}\Lambda^{k+1}(T),
 \quad 
 \cartan^{k} \mathring\calP_{r}\Lambda^{k}(T) \subseteq \mathring\calP_{r-1}\Lambda^{k+1}(T), 
 \\
 \label{math:simplex:polydiffform:diffinclusion:bc}
 \cartan^{k} \calP_{r}^{-}\Lambda^{k}(T) = \cartan^{k} \calP_{r}\Lambda^{k}(T),
 \quad 
 \cartan^{k} \mathring\calP_{r}^{-}\Lambda^{k}(T) = \cartan^{k} \mathring\calP_{r}\Lambda^{k}(T).
\end{gather}
For positive polynomial order $r \geq 1$ we additionally have 
\begin{gather}
 \calP_{r}^{-}\Lambda^{0}(T) = \calP_{r}\Lambda^{0}(T),
 \quad 
 \mathring\calP_{r}^{-}\Lambda^{0}(T) = \mathring\calP_{r}\Lambda^{0}(T),
 \\
 \calP_{r}^{-}\Lambda^{n}(T) = \calP_{r-1}\Lambda^{n}(T),
 \quad 
 \mathring\calP_{r}^{-}\Lambda^{n}(T) = \mathring\calP_{r-1}\Lambda^{n}(T)
 .
\end{gather}
One can show that $\mathring\calP_r\Lambda^{k} = \nullspace$ if $r < n-k-1$
and that $\mathring\calP^{-}_r\Lambda^{k} = \nullspace$ if $r < n-k+1$.
This follows, for example, from the dimension identities 
\begin{subequations}
\begin{gather*}
 \dim \calP_r\Lambda^{k}(T) = \binom{n+r}{n}\binom{n}{k},
 \quad
 \dim \calP^{-}_r\Lambda^{k}(T) = \binom{r+k-1}{k}\binom{n+r}{n-k},
 \\
 \dim \mathring\calP_r\Lambda^{k}(T) = \binom{r+1}{n-k}\binom{r+k}{k},
 \quad 
 \dim \mathring\calP^{-}_r\Lambda^{k}(T) = \binom{n}{k}\binom{r+k-1}{n}.
\end{gather*}
\end{subequations}

\begin{remark}
 Definitions~\eqref{math:space:pr} and~\eqref{math:space:prminus} of $\calP_r\Lambda^{k}(T)$ and $\calP_r^{-}\Lambda^{k}(T)$
 are in terms of spanning sets that are not linearly independent in general.
 For explicit bases for $\calP_r\Lambda^{k}(T)$ and $\calP_r^{-}\Lambda^{k}(T)$
 and the spaces $\mathring\calP_r\Lambda^{k}(T)$ and $\mathring\calP_r^{-}\Lambda^{k}(T)$
 we refer to~\cite{afwgeodecomp}. \end{remark}

Two important polynomial differential forms over $T$
are $1_T \in \calP_0\Lambda^{0}(T)$, the constant function over $T$ which at each point takes the value $1$,
and the volume form $\vol_T \in \calP_0\Lambda^{n}(T)$, the unique constant $n$-form over $T$
with $\int_T \vol_T = \vol^{n}(T)$. These span the constant functions $\calP_0\Lambda^{0}(T)$ and the constant $n$-forms $\calP_0\Lambda^{n}(T)$, respectively.
The former is the kernel of $\cartan^{0}$ and the latter is complementary to the range of $\cartan^{n-1}$. 
It will be convenient to introduce notation for spaces with those special differential forms removed. 
Let $\int_T : C^{\infty}\Lambda^{0}(T) \rightarrow \bbR$ and $\int_T : C^{\infty}\Lambda^{n}(T) \rightarrow \bbR$
denote the respective integral mappings of $0$- and $n$-forms over $T$. 
We set  
\begin{align}
 \underline\calP_{r}\Lambda^{k}(T)
 &:=
 \left\{\begin{array}{rl}
  \calP_{r}\Lambda^{0}(T) \cap \ker \int_T &\textif k = 0,
  \\
  \calP_{r}\Lambda^{k}(T) & \text{ otherwise, }
 \end{array}\right. 
 \\
 \underline\calP_{r}^{-}\Lambda^{k}(T)
 &:=
 \left\{\begin{array}{rl}
  \calP_{r}\Lambda^{0}(T) \cap \ker \int_T &\textif k = 0,
  \\
  \calP_{r}^{-}\Lambda^{k}(T) & \text{ otherwise, }
 \end{array}\right.  
 \\
 \underline{\mathring\calP}_{r}\Lambda^{k}(T)
 &:=
 \left\{\begin{array}{rl}
  \mathring\calP_{r}\Lambda^{n}(T) \cap \int_T &\textif k = n,
  \\
  \mathring\calP_{r}\Lambda^{k}(T) & \text{ otherwise, }
 \end{array}\right. 
 \\
 \underline{\mathring\calP}_{r}^{-}\Lambda^{k}(T)
 &:=
 \left\{\begin{array}{rl}
  \mathring\calP_{r}^{-}\Lambda^{n}(T) \cap \int_T &\textif k = n, 
  \\
  \mathring\calP_{r}^{-}\Lambda^{k}(T) & \text{ otherwise. }
 \end{array}\right. 
\end{align}
Given $r \geq 0$, we obviously have the direct sum decompositions 
\begin{gather}
 \label{math:simplex:polydiffform:decomposition}
 \calP_{r}\Lambda^{0}(T) = \underline\calP_{r}\Lambda^{0}(T) \oplus \bbR \cdot 1_T,
 \quad 
 \calP_{r+1}^{-}\Lambda^{0}(T) = \underline\calP_{r+1}^{-}\Lambda^{0}(T) \oplus \bbR \cdot 1_T,
 \\
 \label{math:simplex:polydiffform:decomposition:bc}
 {\mathring\calP}_{r}\Lambda^{n}(T) = \underline{\mathring\calP}_{r}\Lambda^{n}(T) \oplus \bbR \cdot \vol_T, 
 \quad 
 {\mathring\calP}_{r+1}^{-}\Lambda^{n}(T) = \underline{\mathring\calP}_{r+1}^{-}\Lambda^{n}(T) \oplus \bbR \cdot \vol_T, 
\end{gather}
and no changes in the other cases. 

With these spaces, 
we may concisely state the following exactness properties of complexes of polynomial differential forms:
\begin{gather}
 \label{math:exactnessresult}
 \forall \omega \in \underline\calP_{r}\Lambda^{k}(T)
 : 
 \left( \cartan^{k}\omega = 0 \implies \exists \eta \in \underline\calP_{r+1}^{-}\Lambda^{k-1} : \cartan^{k-1} \eta = \omega \right), 
 \\
 \label{math:exactnessresult:bc}
 \forall \omega \in \underline{\mathring\calP}_{r}\Lambda^{k}(T)
 : 
 \left( \cartan^{k}\omega = 0 \implies \exists \eta \in \underline{\mathring\calP}_{r+1}^{-}\Lambda^{k-1} : \cartan^{k-1} \eta = \omega \right).
\end{gather}
These results will be used in the next section. 

\begin{example}
 We recapitulate a few examples 
 how these concepts translate to classical finite element spaces 
 when $T$ is a triangle and $r \geq 1$. 
 We refer to~\cite{AFW1} for further elaboration on these examples. 
 
 For $k=0$ the space $\calP_{r}\Lambda^{0}(T) = \calP_{r}^{-}\Lambda^{0}(T)$ translates into the space of order $r$ polynomials over $T$.
 Additionally $\mathring\calP_{r}\Lambda^{0}(T) = \mathring\calP_{r}^{-}\Lambda^{0}(T)$ is the subspace satisfying Dirichlet boundary conditions along the edges,
 and $\underline\calP_{r}\Lambda^{0}(T) = \underline\calP_{r}^{-}\Lambda^{0}(T)$ 
 is the subspace of order $r$ polynomials with vanishing mean value. 
 
 For $k=1$ the two families translate into different spaces: 
 $\calP_r\Lambda^{k}(T)$ translates into the order $r$ Brezzi-Douglas-Marini space $\BDM_{r}(T)$, 
 and
 $\calP_{r+1}^{-}\Lambda^{k}(T)$ translates into the order $r$ Raviart-Thomas space $\RT_{r}(T)$.
We write $\mathring\BDM_{r}(T)$ and $\mathring\RT_{r}(T)$ for the subspaces with boundary conditions,
 which in this case are normal boundary conditions along the simplex boundary.
 
 Finally, for $k=2$ we have $\calP_{r-1}\Lambda^{2}(T) = \calP_{r}^{-}\Lambda^{2}(T)$.
 This space translates into polynomials over $T$ of order $r-1$,
 but this time imposing boundary conditions does not change the space. 
 The subspace $\underline\calP_{r-1}\Lambda^{2}(T) = \underline\calP_{r}^{-}\Lambda^{2}(T)$
 corresponds to the order $(r-1)$ polynomials over $T$ with vanishing mean value. 
\end{example}

\section{Polynomial de~Rham Complexes over Simplices} \label{sec:simplexcomplexes}

This section develops a theory of polynomial de~Rham complexes over simplices. 
We prove their exactness and obtain a representation of the degrees of freedom. 
We first observe that differential complexes of similar \emph{type} appear throughout finite element exterior calculus in different variants.
For example, a differential complex of trimmed polynomial differential forms of fixed order $r$ appears as a differential complex over a single simplex,
over a triangulation, or with boundary conditions. 
It is of interest to turn the idea of sequences having a \emph{type} into a rigorous mathematical notion. 
A particular motivation is developing differential complexes in the theory of $hp$-adaptive methods, composed of finite element spaces of non-uniform polynomial order. 
In that application, we wish to assign types of polynomial de~Rham complexes to each simplex to describe the local order of approximation. 
\sectionintrolinebreak

We first introduce a set of formal symbols 
\begin{align}
 \label{math:formalsymbols}
 \scrS
 :=
 \left\{ 
  \ldots, \calP_{r-1}, \calP_{r}^{-}, \calP_r, \calP_{r+1}^{-}, \ldots
 \right\}
 .
\end{align}
The set $\scrS$ is endowed with a total order $\leq$ defined by $\calP_r^{-} \leq \calP_{r}$ and $\calP_{r} \leq \calP_{r+1}^{-}$.

An \emph{admissible sequence type} is a mapping $\calP : \bbZ \rightarrow \scrS$
that satisfies the condition
\begin{align}
 \label{math:admissiblesequencetype}
 \forall k \in \bbZ : 
 \calP(k) \in \left\{ \calP_r^{-}, \calP_r \right\}
 \implies 
 \calP(k+1) \in \left\{ \calP_{r}^{-}, \calP_{r-1} \right\}
 .
\end{align}
If $\calP \in \scrA$ is an admissible sequence type and $T$ is an $n$-simplex,
then we define for each $k \in \bbZ$ the spaces 
\begin{align}
 \calP\Lambda^{k}(T)
 :=
 \left\{ \begin{array}{rl}
  \calP_r\Lambda^{k}(T) 
  &\text{if } \calP(k) = \calP_r,
  \\
  \calP_r^{-}\Lambda^{k}(T)
  &\text{if } \calP(k) = \calP_r^{-},
 \end{array}\right.
 \\ 
 \mathring\calP\Lambda^{k}(T)
 :=
 \left\{ \begin{array}{rl}
  \mathring\calP_r\Lambda^{k}(T) 
  &\text{if } \calP(k) = \calP_r,
  \\
  \mathring\calP_r^{-}\Lambda^{k}(T)
  &\text{if } \calP(k) = \calP_r^{-},
 \end{array}\right.
 \\
 \underline{\calP}\Lambda^{k}(T)
 :=
 \left\{ \begin{array}{rl}
  \underline{\calP}_r\Lambda^{k}(T) 
  &\text{if } \calP(k) = \calP_r,
  \\
  \underline{\calP}_r^{-}\Lambda^{k}(T)
  &\text{if } \calP(k) = \calP_r^{-},
 \end{array}\right.
 \\ 
 \underline{\mathring\calP}\Lambda^{k}(T)
 :=
 \left\{ \begin{array}{rl}
  \underline{\mathring\calP}_r\Lambda^{k}(T) 
  &\text{if } \calP(k) = \calP_r,
  \\
  \underline{\mathring\calP}_r^{-}\Lambda^{k}(T)
  &\text{if } \calP(k) = \calP_r^{-}.
 \end{array}\right.
\end{align}
We let $\scrA$ denote the set of admissible sequence types. 
The total order on $\scrS$ induces a partial order $\leq$ on $\scrA$,
where for all $\calP,\calS \in \scrA$ we have $\calP \leq \calS$
if and only if for all $k \in \bbZ$ we have $\calP(k) \leq \calS(k)$.

The notation already suggests that the symbols $\scrS$ describe finite element spaces,
whereas the admissible sequence types $\scrA$ describe finite element differential complexes. 
To make this idea rigorous, we begin with an easy observation that follows from~\eqref{math:admissiblesequencetype}.
For each admissible sequence type $\calP \in \scrA$, $k \in \bbZ$ and simplex $T$ we have 
\begin{gather*}
 \cartan^{k} \calP\Lambda^{k}(T) \subseteq \calP\Lambda^{k+1}(T), 
 \quad 
 \cartan^{k} \mathring\calP\Lambda^{k}(T) \subseteq \mathring\calP\Lambda^{k+1}(T),
 \\
 \cartan^{k} \underline{\calP}\Lambda^{k}(T) \subseteq \underline{\calP}\Lambda^{k+1}(T),
 \quad 
 \cartan^{k} \underline{\mathring\calP}\Lambda^{k}(T) \subseteq \underline{\mathring\calP}\Lambda^{k+1}(T)
 . 
\end{gather*}
In light of this, we compose differential complexes in accordance with a given admissible sequence type.
Suppose that $T$ is a simplex and that $\calP \in \scrA$ is an admissible sequence type. 
Then we have a polynomial de~Rham complex over $T$, \begin{gather}
 \label{math:polydiffcomplex:absolute}
 \begin{CD}
  0 \to \bbR @>>> \calP\Lambda^{0}(T) @>\cartan^{0}>> \dots @>\cartan^{n-1}>> \calP\Lambda^{n}(T) \to 0 
 \end{CD}
\end{gather}
and a polynomial de~Rham complex over $T$ with boundary conditions, \begin{gather}
 \label{math:polydiffcomplex:relative}
 \begin{CD}
  0 \to \mathring \calP\Lambda^{0}(T) @>\cartan^{0}>> \dots @>\cartan^{n-1}>> \mathring \calP\Lambda^{n}(T) @>>> \bbR \to 0
  .
 \end{CD}
\end{gather}
We will also consider the reduced differential complexes 
\begin{gather}
 \label{math:polydiffcomplex:absolute:reduced}
 \begin{CD}
  0 \to \underline\calP\Lambda^{0}(T) @>\cartan^{0}>> \dots @>\cartan^{n-1}>> \underline\calP\Lambda^{n}(T) \to 0 
 \end{CD}
\\
 \label{math:polydiffcomplex:relative:reduced}
 \begin{CD}
  0 \to \underline{\mathring \calP}\Lambda^{0}(T) @>\cartan^{0}>> \dots @>\cartan^{n-1}>> \underline{\mathring \calP}\Lambda^{n}(T) \to 0
  .
 \end{CD}
\end{gather}
We establish the exactness of these differential complexes. 

\begin{lemma}
 Let $T$ be a simplex and let $\calP \in \scrA$ be an admissible sequence type.
 If $1_T \in \calP\Lambda^{0}(T)$, then~\eqref{math:polydiffcomplex:absolute} is well-defined and exact.
If $\vol_T \in \mathring \calP\Lambda^{n}(T)$, then~\eqref{math:polydiffcomplex:relative} is exact.
\end{lemma}

\begin{proof} 
 With regard to the first sequence, it is obvious that $\ker \cartan^{0} \cap \calP\Lambda^{0}(T)$ is spanned by $1_T$. 
 Let $k \in \{ 1, \dots, n \}$ and $\omega \in \calP\Lambda^{k}(T)$ with $\cartan^{k}\omega = 0$. 
 Then there exists a maximal $r \in \bbZ$ with $\omega \in \calP_r\Lambda^{k}(T) \subseteq \calP\Lambda^{k}(T)$.
 By Lemma~3.8 of~\cite{AFW1}, 
 there exists $\xi \in \calP_{r+1}^{-}\Lambda^{k-1}(T)$ with $\cartan^{k-1} \xi = \omega$.
 Since $\calP_{r+1}^{-}\Lambda^{k-1}(T) \subseteq \calP\Lambda^{k-1}(T)$,
 the exactness of the first sequence follows. 
 
 With regard to the second sequence, it is obvious that $\ker \cartan^{0} \cap \mathring \calP\Lambda^{0}(T)$ is the trivial vector space.
 Now let $k \in \{ 1, \dots, n \}$ and $\omega \in \calP\Lambda^{k}(T)$ with $\cartan^{k}\omega = 0$.
 We assume additionally $\int_T \omega = 0$ if $k = n$.
There exists a maximal $r \in \bbZ$ such that $\omega \in \mathring\calP_r\Lambda^{k}(T) \subseteq \mathring\calP\Lambda^{k}(T)$. 
 Using the commuting interpolant developed in Section~5 of~\cite{AFW1},
we find $\eta \in \mathring\calP^{-}_{r+1}\Lambda^{k-1}(T)$ with $\cartan^{k-1} \eta = \omega$.
 Since $\mathring\calP^{-}_{r+1}\Lambda^{k-1}(T) \subseteq \mathring\calP\Lambda^{k-1}(T)$,
 the exactness of the second sequence follows.
\end{proof}

\begin{lemma}
 Let $T$ be a simplex and let $\calP$ be an admissible sequence type. Then~\eqref{math:polydiffcomplex:absolute:reduced} 
and~\eqref{math:polydiffcomplex:relative:reduced}
are exact sequences. \qed
\end{lemma}

\begin{proof}
 If $1_T \in \calP\Lambda^{0}(T)$, then $\calP\Lambda^{0}(T) = \bbR \cdot 1_T \oplus \underline\calP\Lambda^{0}(T)$,
 and 
 if $\vol_T \in \mathring\calP\Lambda^{n}(T)$, then $\calP\Lambda^{n}(T) = \bbR \cdot \vol_T \oplus \underline{\mathring\calP}\Lambda^{n}(T)$.
 The claim now follows immediately from the preceding result. 
\end{proof}

Now we direct our attention towards dual spaces and their representations.
This prepares the discussion of the degrees of freedom of finite element de~Rham complexes in Section~\ref{sec:hierarchycomplexes}. 
Our approach to the degrees of freedom differs from the approach in~\cite{AFW1} but bears resemblance to the approach in~\cite{deRhamHPFEM}.

Let $T$ be a simplex and let $g$ be any smooth Riemannian metric over $T$. 
This induces a positive definite bilinear form~\cite{FriAgri}
\begin{gather*}
 B_{g} : 
C^{\infty}\Lambda^{k}(T) \times C^{\infty}\Lambda^{k}(T) \rightarrow \bbR,
 \quad 
 ( \omega, \eta ) \mapsto \int_T \langle \omega, \eta \rangle_{g}.
\end{gather*}
The restriction of this bilinear form to any finite-dimensional subspace of $C^{\infty}\Lambda^{k}(T)$ 
gives a Hilbert space structure on that subspace.
We apply this idea to the spaces $\underline{\mathring\calP}\Lambda^{k}(T)$,
since this is the special case needed in later sections. 
The following lemma, however, generalizes to the spaces of the form $\calP\Lambda^{k}(T)$, $\mathring\calP\Lambda^{k}(T)$ and $\underline\calP\Lambda^{k}(T)$
with minimal changes.

\begin{lemma} \label{lemma:dofrepresentation:advanced}
 Let $\calP \in \scrA$. 
 Let $\Psi : \underline{\mathring\calP}\Lambda^{k}(T) \rightarrow \bbR$ be a linear functional. 
 Then there exist $\rho \in \underline{\mathring\calP}\Lambda^{k-1}(T)$ and $\beta \in \underline{\mathring\calP}\Lambda^{k}(T)$
 such that 
 \begin{gather*}
  \Psi(\omega)
  =
  \int_T \langle \omega, \cartan^{k-1} \rho \rangle_{g}
  +
  \int_T \langle \cartan^{k} \omega, \cartan^{k} \beta \rangle_{g},
  \quad 
  \omega \in \underline{\mathring\calP}\Lambda^{k}(T)
  .
 \end{gather*} 
\end{lemma}

\begin{proof}
 Let $\Psi : \underline{\mathring\calP}\Lambda^{k}(T) \rightarrow \bbR$ be linear and let $\omega \in \underline{\mathring\calP}\Lambda^{k}(T)$ be arbitrary. 
 Since $B_g$ induces a Hilbert space structure on a finite-dimensional vector space, 
 the Riesz representation theorem ensures the existence of $\eta \in \underline{\mathring\calP}\Lambda^{k}(T)$
 such that $\Psi(\omega) = B_g(\omega,\eta)$. 
We write $A_0 = \underline{\mathring\calP}\Lambda^{k}(T) \cap \ker\cartan^{k}$ and let $A_1$ denote the orthogonal complement of $A_0$
 in $\underline{\mathring\calP}\Lambda^{k}(T)$ with respect to the scalar product $B_g$. 
 We have an orthogonal decomposition $\underline{\mathring\calP}\Lambda^{k}(T) = A_0 \oplus A_1$,
 and unique decomposition $\omega = \omega_0 + \omega_1$ and $\eta = \eta_0 + \eta_1$
 with $\omega_0, \eta_0 \in A_0$ and $\omega_1, \eta_1 \in A_1$.
 Thus 
 \begin{gather*}
  \Psi(\omega) 
  = 
  \int_T \langle\omega,\eta\rangle_{g} 
  =
  \int_T \langle\omega_{0},\eta_{0}\rangle_{g} + \int_T \langle\omega_{1},\eta_{1}\rangle_{g}
  .
 \end{gather*}
 By the exactness of~\eqref{math:polydiffcomplex:relative:reduced}
 there exists $\rho \in \underline{\mathring\calP}\Lambda^{k-1}(T)$ such that $\eta_0 = \cartan^{k-1} \rho$.
Since the bilinear form $B_{g}\left( \cartan^{k} \cdot, \cartan^{k} \cdot \right)$ is a scalar product 
 over the finite-dimensional space $A_1$, 
 we may use the Riesz representation theorem once again to obtain $\beta \in \underline{\mathring\calP}\Lambda^{k}(T)$ 
 with $B_{g}\left( \cartan^{k} \omega_1, \cartan^{k} \beta \right) = B_{g}\left( \omega_1, \eta_1 \right)$.
 The proof is complete. 
\end{proof}

\section{The Complex of Whitney Forms} \label{sec:whitneycomplex}

In the preceding section, we have studied finite element differential complexes over simplices. 
We now proceed to finite element differential complexes over triangulations.
We begin in this section with the special case of lowest order: the complexes of Whitney forms. 
An important concept is the canonical interpolator. 
\sectionintrolinebreak

Let $\calT$ be a simplicial complex.
This means that $\calT$ is a set of simplices 
\begin{subequations}\label{math:simplicialcomplex}
such that 
\begin{gather}
 \label{math:simplicialcomplex:uno}
 \forall T \in \calT : \forall F \in \Delta(T) : F \in \calT,
 \\
 \label{math:simplicialcomplex:duo}
 \forall T,T' \in \calT : T \cap T' \in \calT \cup \set \emptyset .
\end{gather}
\end{subequations}
In other words, the set of simplices $\calT$ is closed under taking subsimplices
and the intersection of two simplices in $\calT$ is either empty
or a common subsimplex.
We let $\calT^{k}$ denote the set of $k$-simplices in $\calT$.
The simplest example of a simplicial complex is the set of subsimplices $\Delta(T)$ of any simplex $T$. 
Other examples are triangulations of domains. 
A simplicial complex $\calU \subseteq \calT$ is called a simplicial subcomplex of $\calT$.
Note that $\calU = \emptyset$ is permissible. 
\\

For each triangulation, we have an associated simplicial chain complex. 
We recall that we assume the simplices in $\calT$ to be equipped with an arbitrary but fixed orientation.
The space $\calC_{k}(\calT)$ of simplicial $k$-chains is defined as the real vector space 
spanned by the oriented $k$-simplices in $\calT^{k}$.

We recall that the orientation of a simplex $T$ induces an orientation on its subsimplices of one dimension lower. 
When $T \in \calT^{k}$ and $F \in \calT^{k-1}$ with $F \in \Delta(T)$,
then we set $o(F,T) := 1$ if the fixed orientation over $T$ induces the fixed orientation over $F$,
and set $o(F,T) := -1$ in the opposite case. The \emph{simplicial boundary operator} is the linear operator 
\begin{align*}
 \partial_{k} : \calC_{k}(\calT) \rightarrow \calC_{k-1}(\calT) 
\end{align*}
that is defined by taking the linear extension of setting 
\begin{align*}
 \partial_{k} T := \sum_{ F \in \Delta(T)^{k-1} } o(F,T) F, \quad T \in \calT^{k}.
\end{align*}
This operator satisfies the differential property $\partial_{k-1} \partial_{k} = 0$. 
When $\calU \subseteq \calT$ is a simplicial subcomplex then we define the vector space $\calC_{k}(\calT,\calU)$ as the factor space 
\begin{align*}
 \calC_{k}(\calT,\calU) := \calC_{k}(\calT) / \calC_{k}(\calU).
\end{align*}
Note that $\calC_{k}(\calT,\emptyset) = \calC_{k}(\calT)$.
A canonical basis of $\calC_{k}(\calT,\calU)$
is given by (the equivalence classes of) the oriented $k$-simplices in $\calT^{k}$ which are not contained in $\calU^{k}$.
In particular, we can identify $\calC_{k}(\calT,\calU)$ with the subspace of $\calC_m(\calT)$ spanned by $\calT^{k} \setminus \calU^{k}$. 
The simplicial boundary operator induces a well-defined operator 
\begin{align*}
 \partial_{k} : \calC_{k}(\calT,\calU) \rightarrow \calC_{k-1}(\calT,\calU), 
\end{align*}
which again satisfies the differential property $\partial_{k-1} \partial_{k} = 0$.
Accordingly, we introduce the simplicial chain complex 
\begin{align}
 \label{math:chaincomplex}
 \begin{CD}
  \dots
  @<{\partial_{k-1}}<<
  \calC_{k-1}(\calT,\calU)
  @<{\partial_{k  }}<<
  \calC_{k  }(\calT,\calU)
  @<{\partial_{k+1}}<<
  \dots
 \end{CD}
\end{align}
The dimension of the $k$-th homology space of this complex, 
\begin{gather}
 \label{math:chaincohomology}
 b_{k}(\calT,\calU)
 :=
 \dim 
 \dfrac{
  \ker   \partial_{k  } : \calC_{k  }(\calT,\calU) \rightarrow \calC_{k-1}(\calT,\calU)
 }{
  \range \partial_{k+1} : \calC_{k+1}(\calT,\calU) \rightarrow \calC_{k}(\calT,\calU)
 }
 ,
\end{gather}
is known as the \emph{$k$-th simplicial Betti number} of $\calT$ relative to $\calU$.
We call $b_{k}(\calT) := b_{k}(\calT,\emptyset)$
just the \emph{$k$-th simplicial Betti number} of $\calT$.
\\

We now introduce differential forms into the discussion. We define \begin{align*}
C^{\infty}\Lambda^{k}(\calT)
 :=
 \left\{ \;
  (\omega_{T})_{T} \in \bigoplus_{ T \in \calT } C^{\infty}\Lambda^{k}(T)
  \;\suchthat*\; 
  \forall T \in \calT : 
  \forall F \in \Delta(T) : 
  \trace^{k}_{T,F} \omega_T = \omega_F
 \;\right\}
 .
\end{align*}
Via a linear algebraic isomorphism we may identify the space $C^{\infty}\Lambda^{k}(\calT)$ with the space of differential $k$-forms 
that are piecewise smooth with respect to $\calT$ 
and that have single-valued traces along simplex boundaries. 
Our choice of formalism will simplify the notation in what follows.
Henceforth, we may also write $\trace_{T} \omega := \omega_T$ for $\omega \in C^{\infty}\Lambda^{k}(\calT)$ and $T \in \calT$.

Because the exterior derivative commutes with trace operators, 
we have a well-defined exterior derivative on $C^{\infty}\Lambda^{k}(\calT)$ given by 
\begin{align}
 \label{math:smoothforms:cartan}
 \cartan^{k} : 
 C^{\infty}\Lambda^{k}(\calT) \rightarrow C^{\infty}\Lambda^{k+1}(\calT),
 \quad 
 (\omega_T)_{T \in \calT} \mapsto (\cartan^{k}\omega_T)_{T \in \calT}.
\end{align}
Since $\cartan^{k+1} \cartan^{k} \omega = 0$ for every $\omega \in C^{\infty}\Lambda^{k}(\calT)$,
we may compose a differential complex 
\begin{align}
 \label{math:smoothforms:diffcomplex}
 \begin{CD}
  \dots 
  @>\cartan^{k-1}>>
  C^{\infty}\Lambda^{k  }(\calT)
  @>\cartan^{k  }>>
  C^{\infty}\Lambda^{k+1}(\calT)
  @>\cartan^{k+1}>>
  \dots
 \end{CD}
\end{align}
In order to formalize boundary conditions, we furthermore define 
\begin{align}
 \label{math:smoothforms:boundaryconditions}
 C^{\infty}\Lambda^{k}(\calT,\calU)
 :=
 \left\{\; 
  \omega \in C^{\infty}\Lambda^{k}(\calT)
  \;\suchthat*\; 
  \forall F \in \calU : \omega_F = 0
 \;\right\}
 .
\end{align}
It is easily verified that 
\begin{align}
 \cartan^{k} \left( C^{\infty}\Lambda^{k}(\calT,\calU) \right) \subseteq C^{\infty}\Lambda^{k+1}(\calT,\calU).
\end{align}
In particular, we may compose the differential complex 
\begin{align}
 \label{math:smoothforms:diffcomplex:bc}
 \begin{CD}
  \dots 
  @>\cartan^{k-1}>>
  C^{\infty}\Lambda^{k  }(\calT,\calU)
  @>\cartan^{k  }>>
  C^{\infty}\Lambda^{k+1}(\calT,\calU)
  @>\cartan^{k+1}>>
  \dots
 \end{CD}
\end{align}
with abstract boundary conditions. 

\begin{remark} 
 Constructions similar to our definition of $C^{\infty}\Lambda^{k}(\calT)$ have appeared in mathematics before. 
 Our definition is a special case of a \emph{finite element system} in the terminology of~\cite{StructPresDisc}. 
 Another variant is exemplified by \emph{Sullivan forms} in global analysis~\cite{ducret2009lq},
 which are piecewise \emph{flat} differential forms in the sense of geometric measure theory. 
 
 For a practical illustration, suppose that $\Omega \subset \bbR^{n}$ is a bounded Lipschitz domain triangulated by a simplicial complex $\calT$.
 Then the members of $C^{\infty}\Lambda^{k}(\calT)$ correspond to the differential $k$-forms over $\Omega$
 that are piecewise smooth with respect to $\calT$ and have single-valued traces on subsimplices. 
 In addition, suppose that $\Gamma \subset \partial \Omega$ is a subset of the boundary
 and that $\calU$ is a simplicial subcomplex of $\calT$ triangulating $\Gamma$.
 Then $C^{\infty}\Lambda^{k}(\calT,\calU)$ is the subspace of $C^{\infty}\Lambda^{k}(\calT,\calU)$
 whose members have vanishing traces along $\Gamma$. 
 In this manner, $\calU$ may be used to model homogeneous boundary conditions appropriate for $\bfH(\curl)$ and $\bfH(\divergence)$ spaces.
\end{remark}

We investigate an important relation between the simplicial chains and the piecewise smooth differential forms with respect to $\calT$.
Suppose that $\omega \in C^{\infty}\Lambda^{k}(\calT,\calU)$ and $T \in \calT^{k} \setminus \calU^{k}$. We then write $\int_T \omega := \int_{T} \trace^{k}_{T} \omega_T$ for the integral of $\omega$ over $T$.
Taking the linear extension provides a bilinear pairing 
\begin{align}
 \label{math:bilinearpairing}
 C^{\infty}\Lambda^{k}(\calT,\calU)
 \times 
 \calC_{k}(\calT,\calU)
 \rightarrow 
 \bbR,
 \quad 
 (\omega,S) \rightarrow 
 \int_{S} \omega
 .
\end{align}
Moreover, we easily observe (by first considering a single simplex and then taking the linear extension) that 
\begin{gather*}
 \int_{ \partial_{k+1} S } \omega = \int_{S} \cartan^{k} \omega,
 \quad 
 \omega \in C^{\infty}\Lambda^{k}(\calT,\calU), \quad S \in \calC_{k+1}(\calT,\calU).
\end{gather*}
The linear pairing~\eqref{math:bilinearpairing} is degenerate in general: 
there exists $\omega \in C^{\infty}\Lambda^{k}(\calT,\calU)$ such that $\omega \neq 0$ but 
\begin{gather*}
 \int_{S} \omega = 0, \quad S \in \calC_{k}(\calT,\calU).
\end{gather*}
We will identify a differential subcomplex of~\eqref{math:smoothforms:diffcomplex:bc} 
restricting to which in the first variable, that is, $\omega$, makes the bilinear pairing~\eqref{math:bilinearpairing} non-degenerate. 
Specifically, we employ a finite element de~Rham complex of lowest polynomial order.
To begin with, we define the spaces of \emph{Whitney forms} by 
\begin{align*}
 \calW\Lambda^{k}(\calT)
 &:=
 \left\{ \;
  \omega \in C^{\infty}\Lambda^{k}(\calT)
  \;\suchthat*\; 
  \forall T \in \calU : \omega_T \in \calP_{1}^{-}\Lambda^{k}(T)
 \; \right\},
 \\
 \calW\Lambda^{k}(\calT,\calU)
 &:=
 \calW\Lambda^{k}(\calT) 
 \cap
 C^{\infty}\Lambda^{k}(\calT,\calU).
\end{align*}
It is an immediate consequence of definitions that we have a well-defined operator 
\begin{align*}
 \cartan^{k  } : \calW\Lambda^{k}(\calT,\calU) \rightarrow \calW\Lambda^{k+1}(\calT,\calU),
\end{align*}
and consequently the \emph{differential complex of Whitney forms} 
\begin{align}
 \label{math:whitneyformcomplex}
  \begin{CD}
  \dots 
  @>\cartan^{k-1}>>
  \calW\Lambda^{k  }(\calT,\calU)
  @>\cartan^{k  }>>
  \calW\Lambda^{k+1}(\calT,\calU)
  @>\cartan^{k+1}>>
  \dots
 \end{CD}
\end{align}
The notion of Whitney forms was originally motivated by their duality to the simplicial chains. 
This is summarized in the following theorem, which has been proven many times~\cite{christiansen2008construction,StructPresDisc}. 

\begin{theorem} \label{theorem:derhampairing:whitney}
 The bilinear pairing 
 \begin{align}
  \label{math:derhampairing:whitney}
  \calW\Lambda^{k}(\calT,\calU)
  \times
  \calC_k(\calT,\calU)
  \rightarrow
  \bbR,
  \quad
  (\omega,S)
  \mapsto
  \int_S \trace_{S}^{k} \omega
\end{align}
 satisfies 
 \begin{gather*}
  \forall \omega \in \calW\Lambda^{k}(\calT,\calU) : 
  \exists S \in \calC_k(\calT,\calU) : 
  \int_{S} \omega \neq 0, 
  \\ 
  \forall S \in \calC_k(\calT,\calU) : 
  \exists \omega \in \calW\Lambda^{k}(\calT,\calU) : 
  \int_{S} \omega \neq 0, 
 \end{gather*}
 that is, it is non-degenerate. \qed
\end{theorem}

As a consequence of Theorem~\ref{theorem:derhampairing:whitney} we obtain a linear isomorphism 
between $\calC^{k}(\calT,\calU)$ and the dual space of $\calW\Lambda^{k}(\calT,\calU)$,
\begin{align*}
 \calC_k(\calT,\calU)
 \simeq 
 \calW\Lambda^{k}(\calT,\calU)'
 .
\end{align*}
In particular, the differential complex of simplicial chains~\eqref{math:chaincomplex}
is isomorphic to the dual complex of the complex of Whitney forms~\eqref{math:whitneyformcomplex},
and the simplicial boundary operator $\partial_{k+1} : \calC_{k+1}(\calT,\calU) \rightarrow \calC_{k}(\calT,\calU)$
is isomorphic to the dual operator of $\cartan^{k} : \calW\Lambda^{k}(\calT,\calU) \rightarrow \calW\Lambda^{k+1}(\calT,\calU)$.
It follows that the cohomology spaces of the complex of Whitney forms~\eqref{math:whitneyformcomplex}
have the same dimension as 
the cohomology spaces as the corresponding cohomology spaces of the simplicial chain complex~\eqref{math:chaincomplex}.
This dimension is precisely the simplicial Betti number $b_k(\calT,\calU)$.
In summary, 
\begin{gather}
 \label{math:whitneycohomology}
 \dim 
 \dfrac{
  \ker   \cartan^{k  } : \calW\Lambda^{k  }(\calT,\calU) \rightarrow \calW\Lambda^{k+1}(\calT,\calU)
 }{
  \range \cartan^{k-1} : \calW\Lambda^{k-1}(\calT,\calU) \rightarrow \calW\Lambda^{k  }(\calT,\calU)
 }
 =
 b_{k}(\calT,\calU)
 ,
\end{gather}
as follows from~\eqref{math:chaincohomology}.

We are now in a position to provide the canonical finite element interpolator
from the space $C^{\infty}\Lambda^{k}(\calT)$ onto the space $\calW\Lambda^{k}(\calT)$.
We define 
\begin{align}
 I_{\calW}^{k} : 
 C^{\infty}\Lambda^{k}(\calT)
 \rightarrow
 \calW\Lambda^{k}(\calT)
\end{align}
by setting 
\begin{align*}
 \int_{S} I_{\calW}^{k} \omega 
 =
 \int_{S} \omega,
 \quad 
 \omega \in C^{\infty}\Lambda^{k}(\calT),
 \quad 
 S \in \calC_{k}(\calT)
 .
\end{align*}
This is well-defined because of Theorem~\ref{theorem:derhampairing:whitney}. 
This is the identity on Whitney forms, i.e.\ 
\begin{align*}
 I_{\calW}^{k} \omega = \omega,
 \quad 
 \omega \in \calW\Lambda^{k}(\calT)
 .
\end{align*}
The operator $I^{k}_{\calW}$ is local in the sense that for every $C \in \calT$ we have 
\begin{align*}
 \omega_{C} = 0 \implies ( I_{\calW}^{k} \omega )_C = 0.
\end{align*}
By restricting the interpolant to $C^{\infty}\Lambda^{k}(\calT,\calU)$ we obtain a well-defined mapping 
\begin{align*}
 I_{\calW}^{k} : 
 C^{\infty}\Lambda^{k}(\calT,\calU)
 \rightarrow
 \calW\Lambda^{k}(\calT,\calU)
 .
\end{align*}
The interpolation operator commutes with the exterior derivative,
\begin{align*}
 \cartan^{k} I_{\calW}^{k} \omega 
 =
 I_{\calW}^{k+1} \cartan^{k} \omega,
 \quad 
 \omega \in C^{\infty}\Lambda^{k}(\calT), 
\end{align*}
as we verify via 
\begin{align*}
 \int_{S} I_{\calW}^{k+1} \cartan^{k} \omega 
 =
 \int_{S} \cartan^{k} \omega
 =
 \int_{\partial_{k+1} S} \omega
 =
 \int_{\partial_{k+1} S} I_{\calW}^{k} \omega
 =
 \int_{S} \cartan^{k} I_{\calW}^{k} \omega
\end{align*}
for $S \in \calC_{k+1}(\calT)$ and $\omega \in C^{\infty}\Lambda^{k}(\calT)$.
So the diagram
\begin{align*}
 \begin{CD}
  \dots 
  @>{\cartan^{k-1}}>>
  C^{\infty}\Lambda^{k  }(\calT,\calU)
  @>{\cartan^{k  }}>>
  C^{\infty}\Lambda^{k+1}(\calT,\calU)
  @>{\cartan^{k+1}}>>
  \dots
  \\
  @.
  @V{I_{\calW}^{k  }}VV
  @V{I_{\calW}^{k+1}}VV
  @.
  \\
  \dots 
  @>{\cartan^{k-1}}>>
  \calW\Lambda^{k  }(\calT,\calU)
  @>{\cartan^{k  }}>>
  \calW\Lambda^{k+1}(\calT,\calU)
  @>{\cartan^{k+1}}>>
  \dots
 \end{CD}
\end{align*}
commutes. 
In particular, $I^{k}_{\calW}$ is a morphism of differential complexes.

i

\section{Higher Order Finite Element Complexes} \label{sec:hierarchycomplexes}

In this section, we study the structure of higher-order finite element differential complexes over triangulations.
We also construct a global interpolant. 
This combines the theoretical preparations carried out in Section~\ref{sec:simplexcomplexes} and Section~\ref{sec:whitneycomplex}.
\sectionintrolinebreak

Let $\calT$ be a simplicial complex and let $\calU$ be a (possibly empty) subcomplex of $\calT$.
We extend our framework of spaces types and admissible sequence types as follows. 
We let $\calP : \calT \rightarrow \scrA$ be a mapping 
that associates to each simplex $T \in \calT$ an admissible sequence type $\calP_T : \bbZ \rightarrow \scrA$.
We then define 
\begin{align}
 \label{math:piecewisepolydiffform}
 \calP\Lambda^{k}(\calT)
 :=
 \left\{\;
  \omega \in C^{\infty}\Lambda^{k}(\calT)
  \;\suchthat*\; 
  \forall T \in \calT : 
  \omega_T \in \calP_{T}\Lambda^{k}(T)
 \;\right\}
 .
\end{align}
By construction, the exterior derivative preserves this class of differential forms, 
\begin{align}
 \label{math:piecewisepolydiffform:cartan}
 \cartan^{k} \calP\Lambda^{k}(\calT) \subseteq \calP\Lambda^{k+1}(\calT)
 .
\end{align}
In particular, we have a differential complex 
\begin{align}
 \label{math:finiteelementcomplex}
 \begin{CD}
  \dots 
  @>\cartan^{k-1}>>
  \calP\Lambda^{k  }(\calT)
  @>\cartan^{k  }>>
  \calP\Lambda^{k+1}(\calT)
  @>\cartan^{k+1}>>
  \dots 
 \end{CD}
\end{align}
Having associated an admissible sequence type $\calP_T$ to each $T \in \calT$,
we say that the \emph{hierarchy condition holds} 
if \begin{align} \label{math:hierarchycondition}
 \forall T \in \calT : \forall F \in \Delta(T) : \calP_F \leq \calP_T.
\end{align}
We assume the hierarchy condition throughout this section. 
In order to simplify the notation, we will write $\calP\Lambda^{k}(T) := \calP_T\Lambda^{k}(T)$
from here on. 

\begin{remark}
 The general idea of the hierarchy condition is that the polynomial order associated to a simplex
 is at least the polynomial order associated to any subsimplex.
 Imposing such a condition is common in the literature on $hp$ finite element methods~\cite{demkowicz2006computing}.
 Indeed, if $(\calP_T)_{T \in \calT}$ violates the hierarchy condition, 
 then there exists a family of sequence types $(\calS_T)_{T \in \calT}$ 
 that satisfies the hierarchy condition and yields the same space $\calP\Lambda^{k}(\calT)$.
 This is analogous to what is called \emph{minimum rule} in~\cite{deRhamHPFEM}. 
\end{remark}

The geometric decomposition of finite element spaces is a concept of paramount importance. 
To have geometric decompositions at our disposal, we make the additional assumption 
that we are given extension operators between finite element spaces over simplices. 
Specifically, we assume to have linear \emph{local extension operators} 
\begin{align}
 \label{math:localextension}
 \ext^{k}_{F,T} :
 \underline{\mathring\calP}\Lambda^{k}(F)
 \rightarrow 
 {\calP}\Lambda^{k}(T)
\end{align}
for every pair $F \in \Delta(T)$ with $T \in \calT$,
satisfying the following conditions:
for all $F \in \calT$ we have
\begin{subequations}
\label{math:localextension:properties}
\begin{align}
 \label{math:localextension:identity}
 \ext^{k}_{F,F} \omega = \omega,
 \quad 
 \omega \in \underline{\mathring\calP}\Lambda^{k}(F)
 ,
\end{align}
for all $T \in \calT$ with $F \in \Delta(T)$ and $f \in \Delta(F)$, we have 
\begin{align}
 \label{math:localextension:consistent}
 \trace^{k}_{T,F} \ext^{k}_{f,T} = \ext^{k}_{f,F}
\end{align}
and for all $T \in \calT$ with $F,G \in \Delta(T)$ but $F \notin \Delta(G)$, we have 
\begin{align}
 \label{math:localextension:locality}
 \trace^{k}_{T,G} \ext^{k}_{F,T} = 0,
\end{align}
\end{subequations}
Provided such a family of extension operators, 
for each $F \in \calT$ we then define the associated \emph{global extension operator},
\begin{align}
 \label{math:globalextension}
 \Ext_{F}^{k} :
 \underline{\mathring\calP}\Lambda^{k}(F)
 \rightarrow 
 C^{\infty}\Lambda^{k}(\calT),
 \quad 
 \mathring\omega 
 \mapsto 
 \bigoplus_{\substack{ T \in \calT\\ F \in \Delta(T) }} 
 \ext^{k}_{F,T} \mathring\omega 
 .
\end{align}
It follows from~\eqref{math:localextension:consistent} that this mapping indeed takes values in $C^{\infty}\Lambda^{k}(\calT)$.
Moreover, 
\begin{align}
 \label{math:globalextension:finiteelementrange}
 \Ext_{F}^{k} \left( \underline{\mathring\calP}\Lambda^{k}(F) \right)
 \subseteq
 \calP\Lambda^{k}(\calT)
 .
\end{align}
We note that $\Ext_{F}^{k} \omega$ for $\omega \in \underline{\mathring\calP}\Lambda^{k}(F)$ 
vanishes on all simplices of $\calT$ that do not contain $F$ as a subsimplex. 

\begin{example}
 \label{example:extensionoperators}
 Extension operators $\ext^{k}_{F,T}$ with these properties are known in the literature~\cite{afwgeodecomp}. 
 We distinguish whether
 $\underline{\mathring\calP}\Lambda^{k}(F) = \underline{\mathring\calP}_{r}\Lambda^{k}(F)$
 or 
 $\underline{\mathring\calP}\Lambda^{k}(F) = \underline{\mathring\calP}^{-}_{r}\Lambda^{k}(F)$
 for some $r \in \bbN$. 
 What follows is a brief outline. 
 
 Let $T \in \calT^{n}$ and $F \in \Delta(T)^{m}$, 
 and let $\{ v_0^{T}, \dots, v_n^{T} \}$ and $\{ v_0^{F}, \dots, v_m^{F} \}$
 be the respective set of vertices. 
Given $\alpha \in A(m)$,
 we let $\alpha_{F,T} \in A(n)$
 be uniquely defined by requiring $\alpha_{F,T}(j) = \alpha(i)$ if $v^{T}_{j} = v^{F}_{i}$ and requiring $\alpha_{F,T}(j) = 0$ otherwise
 for all $j \in \{ 0, \dots, n \}$ and $i \in \{ 0, \dots, m \}$.  
Furthermore,
 given $\sigma \in \Sigma(a:k,m)$ 
 we let $\sigma_{F,T} \in \Sigma(a:k,n)$
 be uniquely defined by $v^{T}_{\sigma_{F,T}} = v^{F}_{\sigma(i)}$ for $a \leq i \leq k$. 
 
 Now, on the one hand, there exists a well-defined linear operator 
 \begin{gather*}
  \ext_{F,T}^{r,k,-} : \calP_{r}^{-}\Lambda^{k}(F) \rightarrow \calP_{r}^{-}\Lambda^{k}(T)
 \end{gather*}
 which is uniquely defined by 
 \begin{gather*}
  \ext_{F,T}^{r,k,-} \lambda^{\alpha}_{F} \phi^{F}_{\rho} = \lambda^{\alpha_{F,T}}_{T} \phi^{T}_{\rho_{F,T}},
  \quad 
  \alpha \in A(r-1,m), \quad \rho \in \Sigma(0:k,0:m).
 \end{gather*}
 The restriction of $\ext_{F,T}^{r,k,-}$ to $\underline{\mathring\calP}^{-}_{r}\Lambda^{k}(F)$
 provides the required mapping. 
 On the other hand, there exists a well-defined linear operator 
 \begin{gather*}
  \ext_{F,T}^{r,k} : \calP_{r}\Lambda^{k}(F) \rightarrow \calP_{r}\Lambda^{k}(T)
 \end{gather*}
 which is uniquely defined by  
 \begin{gather*}
  \ext_{F,T}^{r,k} \lambda^{\alpha}_{F} \cartan\lambda^{F}_{\sigma}
  =
  \lambda^{\alpha_{F,T}}_{T} \Psi^{\alpha,F,T}_{\sigma_{F,T}},
  \quad 
  \alpha \in A(r,m), \quad \sigma \in \Sigma(1:k,0:m),
 \end{gather*}
 where we have used 
 \begin{gather*}
  \Psi^{\alpha,F,T}_{\sigma_{F,T}}
  :=
  \Psi^{\alpha,F,T}_{\sigma_{F,T}(1)}
  \wedge 
  \dots 
  \wedge 
  \Psi^{\alpha,F,T}_{\sigma_{F,T}(k)},
  \\
  \Psi^{\alpha,F,T}_{i}
  :=
  \cartan\lambda^{T}_{i} - \dfrac{\alpha_{F,T}(i)}{r} \sum_{j = 0}^{m} \cartan\lambda^{T}_{\imath_{F,T}(j)},
  \quad 
  0 \leq i \leq n.
\end{gather*}
 The restriction of $\ext_{F,T}^{r,k}$ to $\underline{\mathring\calP}_{r}\Lambda^{k}(F)$
 provides the required mapping. 
\end{example}

The local extension operators enable formalizing the geometric decomposition of $\calP\Lambda^{k}(\calT,\calU)$.
The hierarchy condition is crucial for this endeavor.

Consider $\omega \in \calP\Lambda^{k}(\calT)$.
We define $\omega^{\calW} \in \calP\Lambda^{k}(\calT)$ by 
\begin{align}
 \label{math:geodecomp:whitneypart}
 \omega^{\calW}
 := 
 \sum_{ F \in \calT^{k}} \vol(F)^{\inv} \left( \int_{F} \trace^{k}_{F} \omega \right) \Ext_{F}^{k} \vol_F.
\end{align}
We then define recursively for every $m \in \{ k, \dots, n \}$:
\begin{align}
 \label{math:geodecomp:simplexpart}
 \mathring\omega_{F}
 &:= 
 \trace^{k}_{F}
 \left( 
  \omega - \omega^{\calW} - \sum_{l = k}^{m-1} \omega^{l} 
 \right),
 \quad 
 F \in \calT^{m}, 
 \\
 \label{math:geodecomp:levelpart}
 \omega^{m}
 &:=
 \sum_{ F \in \calT^{m}} 
 \Ext_{F}^{k} \mathring\omega_{F}.
\end{align}
The following theorem shows that these definitions are well-defined and give a decomposition of $\omega$.

\begin{example}
 We give an example of the above construction in the case 
 of the higher-order Lagrange space $\calP_{r}(T)$ over a triangle $T$,
 which is spanned by the barycentric monomials 
 $\{ \lambda_{T}^{\alpha} \}_{ \alpha \in A(r,n) }$ of order $r$. 
 Let $\omega \in \calP_{r}(T)$. 
 Taking point evaluations at the vertices $v^{T}_{0}, v^{T}_{1}, v^{T}_{2}$ of $T$, 
 we define 
 \begin{align*}
  \omega^{\calW}
  = 
  \omega(v^{T}_0) \lambda^{T}_{0} + \omega(v^{T}_1) \lambda^{T}_{1} + \omega(v^{T}_2) \lambda^{T}_{2}.
 \end{align*}
 In particular, $\omega^{\calW}$ is spanned by the basis functions associated to vertices.
 Furthermore, for the three edges $e^{T}_{01}, e^{T}_{02}, e^{T}_{12}$ of $T$ we define 
 \begin{align*}
  \mathring\omega_{e^{T}_{ij}}
  = 
  \trace_{e^{T}_{ij}}
  \left( 
   \omega - \omega^{\calW}
  \right),
  \quad 
  0 \leq i < j \leq 2,
  \\
  \omega^{1}
  =
  \Ext_{e^{T}_{01}} \mathring\omega_{e^{T}_{01}}
  + 
  \Ext_{e^{T}_{02}} \mathring\omega_{e^{T}_{02}}
  + 
  \Ext_{e^{T}_{12}} \mathring\omega_{e^{T}_{12}}
  .
 \end{align*}
 Here, the extension operators map the ``bubble functions'' over the edges to functions over the whole of $T$
 such that the extensions have vanishing traces along the other edges.
 In particular, $\omega^{1}$ is spanned by the basis functions associated to edges.
 Lastly, 
 \begin{align*}
  \omega^{2}
  =
  \mathring\omega_{T}
  = 
  \omega - \omega^{\calW} - \omega^{1} 
 \end{align*}
 is the span of the basis functions associated to the triangle itself. 
 We see that $\omega = \omega^{\calW} + \omega^{1} + \omega^{2}$
 is a decomposition into ``bubbles'' associated to vertices, edges, and the triangle. 
\end{example}

\begin{theorem} \label{theorem:geometricdecomposition}
 Let $\omega \in \calP\Lambda^{k}(\calT)$.
 Then we have $\mathring\omega_{F} \in \underline{\mathring\calP}\Lambda^{k}(F)$
 for every $F \in \calT$
 and 
 \begin{align}
  \label{math:geometricdecomposition}
  \omega 
  =
  \omega^{\calW} 
  +
  \sum_{ m = k }^{n} \omega^{m} 
  .
 \end{align}
\end{theorem}

\begin{proof}
 By construction of $\omega^{\calW}$, we have 
 \begin{gather*}
  \int_{F} \trace^{k}_{F} \omega^{\calW} = \int_{F} \trace^{k}_{F} \omega, \quad F \in \calT^{k}.
 \end{gather*}
 By definition, $\trace_{F}^{k} \left( \omega - \omega^{\calW} \right) \in \underline{\mathring\calP}\Lambda^{k}(F)$ for every $F \in \calT^{k}$.
 With $\omega^{k}$ as defined above, we see  
 \begin{gather*}
  \trace^{k}_{F} \left( \omega - \omega^{\calW} - \omega^{k} \right) = 0,
  \quad F \in \calT^{k}.
 \end{gather*}
 Let us now suppose that for some $m \in \{ k, \dots, n-1 \}$ we have shown 
 \begin{align*}
  \trace^{k}_{f}
  \left( 
   \omega 
   -
   \omega^{\calW} 
   -
   \sum_{ l = k }^{m} 
   \omega^{l}
  \right)
  = 
  0,
  \quad 
  f \in \calT^{m}.
 \end{align*}
 By definition we have $\underline{\mathring\calP}\Lambda^{k}(F) = {\mathring\calP}\Lambda^{k}(F)$ for $F \in \calT^{m+1}$,
 and $\mathring\omega_{F} \in {\mathring\calP}\Lambda^{k}(F)$ for $F \in \calT^{m+1}$. 
We conclude that $\omega^{m+1}$ is well-defined and that  
 \begin{align*}
  \trace^{k}_{F}
  \left( 
   \omega 
   -
   \omega^{\calW} 
   -
   \sum_{ l = k }^{m+1} 
   \omega^{l}
  \right)
  = 
  0,
  \quad 
  F \in \calT^{m+1}.
 \end{align*}
 An induction argument then provides~\eqref{math:geometricdecomposition}.
 The proof is complete. 
\end{proof}

\begin{lemma} \label{lemma:geometricdecomposition:zerocondition}
 Let $\omega \in \calP\Lambda^{k}(\calT)$ and $F \in \calT$.
 Then we have $\omega_F = 0$ if and only if 
\begin{gather*}
  \trace^{k}_{f} \omega^{\calW} = 0, \quad f \in \Delta(F), 
  \\
  \mathring\omega_{f} = 0, \quad f \in \Delta(F)^{k}.
 \end{gather*}
\end{lemma}

\begin{proof}
 For any $\omega \in \calP\Lambda^{k}(\calT)$ and $F \in \calT^{m}$ we observe 
 \begin{align*}
  \omega_{F} 
  &=
  \trace^{k}_{F} \omega^{\calW} 
  +
  \sum_{ k \leq m \leq n} 
  \sum_{ f \in \calT^{m} }
  \trace^{k}_{F} \Ext_{f,\calT}^{k} \mathring\omega_{f}
  \\&=
  \sum_{ f \in \Delta(F)^{k}} \vol(F)^{\inv} \left( \int_{f} \trace^{k}_{f} \omega \right) \Ext_{f,F}^{k} \vol_F
  +
  \sum_{ f \in \Delta(F) }
  \Ext_{f,F}^{k} \mathring\omega_{f}
  .
 \end{align*}
 If $k = m$, then $\omega_F = \trace^{k}_F \omega^{\calW} + \mathring\omega_{F}$, 
 and the claims follow by this being a direct sum.
 If $k < m$, let us assume that the claim holds true for all $f \in \calT$ with $k \leq \dim f < m$.
 Then $\omega_F = \mathring\omega_F$, which again proves the claim. 
 The lemma now follows from an induction argument. 
\end{proof}

\begin{lemma} \label{lemma:geometricdecomposition:boundary}
 For $\omega \in \calP\Lambda^{k}(\calT)$ we have $\omega \in \calP\Lambda^{k}(\calT,\calU)$
 if and only if
 \begin{gather*}
  \mathring\omega_F = 0, \quad F \in \calU,
  \\
  \omega^{\calW}_{F} = 0, \quad F \in \calU^{k}.
 \end{gather*}
\end{lemma}

\begin{proof}
 This is a simple consequence of Lemma~\ref{lemma:geometricdecomposition:zerocondition}.
\end{proof}

\begin{lemma} \label{lemma:geometricdecomposition:globalzero}
 For $\omega \in \calP\Lambda^{k}(\calT)$ we have $\omega = 0$
 if and only if
 \begin{gather*}
  \mathring\omega_F = 0, \quad F \in \calT,
  \\
  \omega^{\calW}_{F} = 0, \quad F \in \calT^{k}.
 \end{gather*}
\end{lemma}

\begin{proof}
 This follows from Lemma~\ref{lemma:geometricdecomposition:boundary} applied to the case $\calU = \calT$.
\end{proof}

\begin{theorem}
 We have 
 \begin{gather*}
  \calP\Lambda^{k}(\calT,\calU)
  = 
  \calW\Lambda^{k}(\calT,\calU) 
  \oplus
  \bigoplus_{ F \in \calT \setminus \calU } \Ext_{F}^{k} \underline{\mathring\calP}\Lambda^{k}(F). 
 \end{gather*}
\end{theorem}

We study a modification of the geometric decomposition.

\begin{lemma} \label{lemma:geometricdecomposition:modified}
 Let $\omega \in \calP\Lambda^{k}(\calT)$.
 Then there exist unique $\mathring\omega_{F} \in \underline{\mathring\calP}\Lambda^{k}(F)$ for $F \in \calT$
 such that 
 \begin{align} \label{math:geometricdecomposition:modified}
  \omega 
  =
  I^{k}_{\calW} \omega 
  +
  \sum_{ k \leq m \leq n} 
  \sum_{ F \in \calT^{m} }
  \Ext_{F}^{k} \mathring\omega^{m}_{F}
  .
 \end{align}
\end{lemma}

\begin{proof}
 Let $\omega \in \calP\Lambda^{k}(\calT)$. 
 The trace of $I^{k}_{\calW} \omega - \omega$ over any simplex $F \in \calT^{k}$ has vanishing integral.
 The claim follows from applying Theorem~\ref{theorem:geometricdecomposition} to $I^{k}_{\calW} - \omega$.
\end{proof}

In the remainder of this section, we define the canonical finite element interpolant and study some of its properties. 
We modify and build upon basic ideas found in the published literature~\cite{deRhamHPFEM}. 
Our construction explicitly calculates the geometric decomposition of the interpolating differential form
and a choice (generally arbitrary) of a Riemannian metric over each simplex. 
We first define 
\begin{gather}
 \label{math:interpolator:whitney}
 J^{k}_{\calW} : C^{\infty}\Lambda^{k}(\calT) \rightarrow \calP\Lambda^{k}(\calT),
 \quad 
 \omega \mapsto \sum_{ F \in \calT^{k} } \vol(F)^{\inv} \left( \int_F \trace^{k}_{F} \omega \right) \Ext_{F}^{k} \vol_{F}
 .
\end{gather}
In what follows, we abbreviate 
\begin{align*}
     \omega^{\calW} := J^{k}_{\calW} \omega.
\end{align*}
Subsequently, for $m \in \{k, \dots, n \}$, we define recursively 
\begin{gather}
 \label{math:interpolator:level}
 J^{k}_{m} : C^{\infty}\Lambda^{k}(\calT) \rightarrow \calP\Lambda^{k}(\calT),
 \quad 
 \omega \mapsto \sum_{ F \in \calT^{m} } \Ext_{F}^{k} J^{k}_{F} \omega, 
\end{gather}
where for each $F \in \calT^{m}$ we define 
\begin{gather}
 \label{math:interpolator:simplex}
 J^{k}_{F} : C^{\infty}\Lambda^{k}(\calT) \rightarrow \underline{\mathring\calP}\Lambda^{k}(F)
\end{gather}
by requiring $J^{k}_{F} \omega$ for $\omega \in C^{\infty}\Lambda^{k}(\calT)$
to be the unique solution of 
\begin{subequations}
\label{math:geodecomp}
\begin{gather}
 \label{math:geodecomp:cycles} 
 \int_{F} \left\langle J^{k}_{F} \omega, \cartan^{k-1} \rho \right\rangle_{g}
 =
\int_{F} \left\langle \trace^{k}_{F}\left( \omega - J^{k}_{\calW} \omega - \sum_{k = l}^{m-1} J^{k}_{l} \omega \right), \cartan^{k-1} \rho \right\rangle_{g},
 \quad 
 \rho \in \underline{\mathring\calP}\Lambda^{k-1}(F),
 \\
 \label{math:geodecomp:cocycles}
 \int_{F} \left\langle \cartan^{k} J^{m}_{F} \omega, \cartan^{k} \beta \right\rangle_{g}
 =
 \int_{F} \left\langle \cartan^{k} \trace^{k}_{F}\left( \omega - J^{k}_{\calW} \omega - \sum_{k = l}^{m-1} J^{k}_{l} \omega \right), \cartan^{k} \beta \right\rangle_{g},
 \quad 
 \beta \in \underline{\mathring\calP}\Lambda^{k}(F).
\end{gather}
\end{subequations}
That $J^{k}_{F} \omega$ is well-defined follows easily from Lemma~\ref{lemma:dofrepresentation:advanced}.
We then set 
\begin{gather}
 \label{math:interpolator:gesamt}
 I^{k}_{\calP} : C^{\infty}\Lambda^{k}(\calT) \rightarrow \calP\Lambda^{k}(\calT),
 \quad 
 \omega 
 \mapsto 
 J^{k}_{\calW} \omega + J^{k}_{k} \omega + \dots  + J^{k}_{n} \omega
 .
\end{gather}
We show that the operator $I^{k}_{\calP}$ acts as the identity on $\calP\Lambda^{k}(\calT)$,
and its constituents $J^{k}_{F}$ reproduce the geometric decomposition.

\begin{lemma} \label{lemma:interpolator:geodecomp}
 For each $\omega \in \calP\Lambda^{k}(\calT)$ we have $I^{k}_{\calP} \omega = \omega$.
 Moreover, $J^{k}_{\calW} \omega = \omega^{\calW}$ and $J^{k}_{F}\omega = \mathring\omega_{F}$ for each $F \in \calT$.
\end{lemma}

\begin{proof}
 Let $\omega \in \calP\Lambda^{k}(\calT)$.
 We have $J^{k}_{\calW} \omega = \omega^{\calW}$ by definition.
 For $F \in \calT^{k}$,
 we find $\trace^{k}_{F}\left( \omega - \omega^{\calW} \right) \in \underline{\mathring\calP}\Lambda^{k}(F)$,
 and $J^{k}_{F} \omega = \mathring\omega_F$ follows easily.
Next, let $m \in \{ k, \dots, n-1 \}$ and suppose that $J^{k}_{F}\omega = \mathring\omega_{F}$ for $F \in \calT$ with $\dim F \leq m$.
 Let $F \in \calT^{m+1}$. 
 Unfolding definitions, we find 
 \begin{gather*}
  \trace^{k}_{F}\left( \omega - \omega^{\calW} - \sum_{l = k}^{m-1} J^{l} \omega \right)
  \in
  \underline{\mathring\calP}\Lambda^{k}(F).
 \end{gather*}
 It follows that $J^{k}_{F} \omega = \mathring\omega_{F}$ and hence $J^{k}_{m} \omega = \omega^{m}$.
 An induction argument completes the proof. 
\end{proof}

\begin{lemma} \label{lemma:interpolater:alternative:zero}
 Let $\omega \in \calP\Lambda^{k}(\calT)$. If
 \begin{subequations}
 \label{math:interpolator:alternative:zero}
 \begin{gather}
  \label{math:interpolator:alternative:zero:whitney}
  \int_{F} \trace^{k}_{F} \omega'
  =
  0,
  \quad
  F \in \calT^{k}, 
  \\
  \label{math:interpolator:alternative:zero:cycle}
  \int_{F} \left\langle \trace^{k}_{F} \omega', \cartan^{k-1} \rho \right\rangle_{g}
  =
  0,
  \quad 
  \rho \in \underline{\mathring\calP}\Lambda^{k-1}(F), \quad F \in \calT, 
  \\
  \label{math:interpolator:alternative:zero:cocycle}
  \int_{F} \left\langle \cartan^{k} \trace^{k}_{F} \omega', \cartan^{k} \beta \right\rangle_{g}
  =
  0,
  \quad 
  \beta \in \underline{\mathring\calP}\Lambda^{k}(F) \quad F \in \calT,
 \end{gather}
 \end{subequations}
 then $\omega = 0$. 
\end{lemma}

\begin{proof}
 This follows from~\eqref{lemma:interpolator:geodecomp} and an induction argument. 
\end{proof}

An auxiliary results yields an alternative characterization of $I^{k}_{\calP}$.

\begin{lemma} \label{lemma:interpolater:alternative}
 Let $\omega \in C^{\infty}\Lambda^{k}(\calT)$ and $\omega' \in \calP\Lambda^{k}(\calT)$. 
 We have $\omega' = I^{k}_{\calP} \omega$ if and only if 
 \begin{subequations}
 \label{math:interpolator:alternative}
 \begin{gather}
  \label{math:interpolator:alternative:whitney}
  \int_{F} \trace^{k}_{F} \omega'
  =
  \int_{F} \trace^{k}_{F} \omega,
  \quad
  F \in \calT^{k}, 
  \\
  \label{math:interpolator:alternative:cycle}
  \int_{F} \left\langle \trace^{k}_{F} \omega', \cartan^{k-1} \rho \right\rangle_{g}
  =
  \int_{F} \left\langle \trace^{k}_{F} \omega, \cartan^{k-1} \rho \right\rangle_{g},
  \quad 
  \rho \in \underline{\mathring\calP}\Lambda^{k-1}(F), \quad F \in \calT, 
  \\
  \label{math:interpolator:alternative:cocycle}
  \int_{F} \left\langle \cartan^{k} \trace^{k}_{F} \omega', \cartan^{k} \beta \right\rangle_{g}
  =
  \int_{F} \left\langle \cartan^{k} \trace^{k}_{F} \omega, \cartan^{k} \beta \right\rangle_{g},
  \quad 
  \beta \in \underline{\mathring\calP}\Lambda^{k}(F) \quad F \in \calT. 
 \end{gather}
 \end{subequations}
\end{lemma}

\begin{proof}
 Let $\omega \in C^{\infty}\Lambda^{k}(\calT)$.
 We verify that $I^{k}_{\calP}\omega$ satisfies~\eqref{math:interpolator:alternative} by rearranging the terms in~\eqref{math:interpolator:simplex}
 and the assumptions on the extension operators. 
 If $\omega' \in \calP\Lambda^{k}(\calT)$ is another solution to~\eqref{math:interpolator:alternative},
 then we obtain $\omega' = I^{k}_{\calP} \omega$ by applying Lemma~\ref{lemma:interpolater:alternative:zero} to $\omega' - I^{k}_{\calP}\omega$. 
 The claim follows by an induction argument. 
\end{proof}

\begin{lemma} \label{lemma:interpolator:locality}
 Let $\omega \in C^{\infty}\Lambda^{k}(\calT)$ and $F \in \calT$. 
 If $\omega_{F} = 0$, then $\trace^{k}_{F} \left( I^{k}_{\calP} \omega \right) = 0$.
\end{lemma}

\begin{proof}
 Unfolding definitions, we find 
 \begin{align*}
  \trace^{k}_{F} \left( I^{k}_{\calP} \omega \right)
  &=
  \trace^{k}_{F} J^{k}_{\calW} \omega 
  +
  \sum_{ m = k}^{n} 
  \sum_{ f \in \calT^{m} }
  \trace^{k}_{F} \Ext_{f,\calT}^{k} J^{k}_{f} \omega
  \\&=
  \sum_{ f \in \Delta(F)^{k}} \vol(F)^{\inv} \left( \int_{f} \trace^{k}_{f} \omega \right) \Ext_{f,F}^{k} \vol_F
  +
  \sum_{ f \in \Delta(F) }
  \Ext_{f,F}^{k} J^{k}_{f} \omega
  .
 \end{align*}
 If $\dim F = k$, then the claim follows from the direct sum decomposition~\eqref{math:simplex:polydiffform:decomposition:bc}.
 If $\dim F > k$, suppose that the claim has been proven for $f \in \Delta(F)$.
 Since $\omega_F = 0$, we have $\omega_f = 0$ for $f \in \Delta(F)$. 
 Hence $\trace^{k}_{F} \left( I^{k}_{\calP} \omega \right) = J^{k}_{F} \omega$,
 from which $\trace^{k}_{F} \left( I^{k}_{\calP} \omega \right) = 0$ follows. 
 An induction argument proves the claim. 
\end{proof}

\begin{lemma} \label{lemma:interpolator:boundarycondition}
 If $\omega \in C^{\infty}\Lambda^{k}(\calT,\calU)$, then $I^{k}_{\calP} \omega \in \calP\Lambda^{k}(\calT,\calU)$.
\end{lemma}

\begin{proof}
 This is an immediate consequence of Lemma~\ref{lemma:interpolator:locality} above. 
\end{proof}

It remains to show that the canonical interpolant commutes with the exterior derivative,
so we have a commuting diagram
\begin{align*}
 \begin{CD}
  \dots 
  @>\cartan^{k-1}>>
  \calC^{\infty}\Lambda^{k  }(\calT,\calU)
  @>\cartan^{k  }>>
  \calC^{\infty}\Lambda^{k+1}(\calT,\calU)
  @>\cartan^{k+1}>>
  \dots
  \\
  @.
  @V{I^{k  }_{\calP}}VV
  @V{I^{k+1}_{\calP}}VV
  @.
  \\
  \dots 
  @>\cartan^{k-1}>>
  \calP\Lambda^{k  }(\calT,\calU)
  @>\cartan^{k  }>>
  \calP\Lambda^{k+1}(\calT,\calU)
  @>\cartan^{k+1}>>
  \dots
 \end{CD}
\end{align*}
This is the subject of the following lemma.

\begin{lemma} \label{lemma:interpolatorcommutes}
 We have $\cartan^{k  } I^{k}_{\calP} \omega = I^{k+1}_{\calP} \cartan^{k  } \omega$ for $\omega \in C^{\infty}\Lambda^{k}(\calT)$.
\end{lemma}

\begin{proof}
 Let $\omega \in C^{\infty}\Lambda^{k}(\calT,\calU)$.
 For $F \in \calT^{k+1}$ we observe 
 \begin{align*}
  \int_F \trace^{k+1}_{F} \cartan^{k  } I^{k}_{\calP} \omega
  &=
  \int_F \trace^{k+1}_{F} \cartan^{k  } J^{k}_{\calW} \omega
  =
  \int_F \cartan^{k  } \trace^{k}_{F} J^{k}_{\calW} \omega
  \\&=
  \int_{\partial F} \trace^{k}_{F} J^{k}_{\calW} \omega 
  =
  \int_{\partial F} \trace^{k}_{F} \omega
  \\&=
  \int_{F} \cartan^{k  } \trace^{k}_{F} \omega 
  =
  \int_{F} \trace^{k+1}_{F} \cartan^{k  } \omega 
  =
  \int_{F} \trace^{k+1}_{F} J^{k+1}_{\calW} \cartan^{k  } \omega 
  =
  \int_{F} \trace^{k+1}_{F} I^{k+1}_{\calP} \cartan^{k  } \omega 
  .
 \end{align*}
 Let $F \in \calT^{m}$ with $k \leq m \leq n$.
 For $\rho \in \mathring\calP\Lambda^{k}(F)$ we find  
 \begin{align*}
  \int_{F} \left\langle I^{k+1}_{\calP} \cartan^{k} \omega, \cartan^{k} \rho \right\rangle_{g} 
  &=
  \int_{F} \left\langle \cartan^{k} \omega, \cartan^{k} \rho \right\rangle_{g} 
  \\&=
  \int_{F} \left\langle \cartan^{k}  I^{k}_{\calP} \omega, \cartan^{k} \rho \right\rangle_{g} 
  =
  \int_{F} \left\langle \cartan^{k}  I^{k}_{\calP} \omega, \cartan^{k} \rho \right\rangle_{g} 
  . 
 \end{align*}
 For $\beta \in \mathring\calP\Lambda^{k+1}(F)$ we find 
 \begin{align*}
  \int_{F} \left\langle \cartan^{k+1}  I^{k+1}_{\calP} \cartan^{k} \omega, \cartan^{k+1} \beta \right\rangle_{g} 
  &=
  \int_{F} \left\langle \cartan^{k+1} \cartan^{k} \omega, \cartan^{k+1} \beta \right\rangle_{g} 
  \\&=
  \int_{F} \left\langle  \cartan^{k+1} \cartan^{k} I^{k}_{\calP} \omega, \cartan^{k+1} \beta \right\rangle_{g} 
  =
  0
  . 
 \end{align*}
 In conjunction with Lemma~\ref{lemma:interpolater:alternative}, the desired result follows. 
\end{proof}

\section{Partially Localized Flux Reconstruction} \label{sec:fluxreconstruction}

In this section, we approach the main result of this article. 
Our investigations on the structure of finite element spaces allow us to formalize a partially localized method of flux reconstruction.
The flux reconstruction solves the first-order differential equation $\cartan^{k-1}\xi = \omega$, 
where $\omega \in \calP\Lambda^{k}(\calT,\calU)$ is the data and $\xi \in \calP\Lambda^{k-1}(\calT,\calU)$ is the unknown. 
Assuming that a solution exists, we wish to efficiently compute one of the possible solutions.
Problems of this type appear in a~posteriori error estimation. 

The problem of flux reconstruction amounts to determining a generalized inverse 
of the operator $\cartan^{k-1} : \calP\Lambda^{k-1}(\calT,\calU) \rightarrow \calP\Lambda^{k}(\calT,\calU)$.
In this article we contribute a method to reduce this problem to the lowest-order case. 
It then only remains to find a generalized inverse of 
$\cartan^{k-1} : \calW\Lambda^{k-1}(\calT,\calU) \rightarrow \calW\Lambda^{k}(\calT,\calU)$.
The higher-order aspects of the problem are treated in local problems associated to simplices 
which are solved independently from each other.
This is a fundamental result on the structure of higher-order finite element spaces
that is not only of theoretical appeal but also relevant for numerical algorithms. 

Before we formulate the main result,
we introduce several generalized inverses. 
First, we fix a generalized inverse of the exterior derivative between Whitney forms. 
Specifically, we assume that we have a linear mapping 
\begin{align}
 \label{math:antiderivative:whitney}
 \sfP_{\calW}^{k} :
 \calW\Lambda^{k}(\calT,\calU) 
 \rightarrow
 \calW\Lambda^{k-1}(\calT,\calU)
\end{align}
such that 
\begin{align}
 \label{math:antiderivative:whitney:condition}
 \cartan^{k-1} \sfP_{\calW}^{k} \cartan^{k-1} \xi = \cartan^{k-1} \xi,
 \quad 
 \xi \in \calW\Lambda^{k-1}(\calT,\calU)
 .
\end{align}
In particular, $\omega = \cartan^{k-1} \sfP_{\calW}^{k} \omega$
whenever $\omega \in \calW\Lambda^{k}(\calT,\calU)$
is the exterior derivative of a Whitney form in $\calW\Lambda^{k-1}(\calT,\calU)$.
Similarly, for each simplex $F \in \calT$ we fix a generalized inverse
\begin{align}
 \label{math:antiderivative:simplex}
 \sfP_{F}^{k}
 :
 \underline{\mathring\calP}\Lambda^{k}(F)
 \rightarrow 
 \underline{\mathring\calP}\Lambda^{k-1}(F)
\end{align}
such that
\begin{align}
 \label{math:antiderivative:simplex:condition}
 \cartan^{k-1} \sfP_{F}^{k} \cartan^{k-1} \xi = \cartan^{k-1} \xi,
 \quad 
 \xi \in \underline{\mathring\calP}\Lambda^{k}(F).
\end{align}
We have $\omega = \cartan^{k-1} \sfP_{F}^{k} \omega$
whenever $\omega \in \underline{\mathring\calP}\Lambda^{k}(F)$
is the exterior derivative of a Whitney form in $\underline{\mathring\calP}\Lambda^{k-1}(F)$.
The existence of a mapping $\sfP_{\calW}^{k}$ and mappings $\sfP_{F}^{k}$ with such properties is elementary.

\begin{remark}
  There is no canonical choice of the generalized inverses. 
  However, upon fixing a Hilbert space structure on the Whitney forms, 
  the Moore-Penrose pseudoinverse of $\cartan^{k-1} : \calW\Lambda^{k-1}(\calT,\calU) \rightarrow \calW\Lambda^{k}(\calT,\calU)$
  is a natural choice that provides the least-squares solution of the problem. 
  Entirely analogous statements hold for choosing the generalized inverses $\sfP^{k}_{F}$. 
\end{remark}

Assuming to have fixed generalized inverses as above,
we provide the partially localized flux reconstruction without further ado.

\begin{theorem} \label{theorem:fluxreconstruction}
 Suppose that $\omega \in \calP\Lambda^{k}(\calT)$ with $\cartan^{k} \omega = 0$.
 For $m \in \{k,\dots,n\}$ we let 
 \begin{gather} \label{theorem:fluxreconstruction:higher}
  \xi^{m} 
  :=
  \sum_{ F \in \calT^{m} }
  \Ext_{F}^{k-1}
  \sfP_{F}^{k}
  \trace^{k}_{F}
  \left( 
   \omega - I^{k}_{\calW} \omega - \sum_{l=k}^{m-1} \cartan^{k-1} \xi^{l}
  \right)
  .
 \end{gather}
 Then
 \begin{gather} \label{theorem:fluxreconstruction:final}
  I^{k}_{\calW} \omega + \cartan^{k-1} \left( \sum_{m=k}^{n} \xi^{m} \right)
  =
  \omega
  .
 \end{gather}
 If there exists $\xi \in \calP\Lambda^{k-1}(\calT,\calU)$ with $\cartan^{k-1} \xi = \omega$, 
 then 
 \begin{gather} \label{theorem:fluxreconstruction:reallyfinal}
  \cartan^{k-1} \left( \sfP_{\calW}^{k} I^{k}_{\calW} \omega + \sum_{m=k}^{n} \xi^{m} \right)
  =
  \omega
  .
 \end{gather}
\end{theorem}

\begin{proof}
 We use the modified geometric decomposition (Lemma~\ref{lemma:geometricdecomposition:modified}) to write 
 \begin{gather*}
  \omega = I^{k}_{\calW} \omega + \sum_{m=k}^{n} \sum_{F \in \calT^{m}} \Ext_{F}^{k} \mathring\omega_{F}
  , 
 \end{gather*}
 where $\mathring\omega_F \in \underline{\mathring\calP}\Lambda^{k}(F)$ for each $F \in \calT$.
 We thus find for $F \in \calT^{k}$ that 
 \begin{gather*}
  \trace^{k}_{F} \left( \omega - I^{k}_{\calW} \omega \right)
  \in
  \underline{\mathring\calP}\Lambda^{k}(F).
 \end{gather*}
 The proof will be completed by an induction argument. 
 For each $F \in \calT$ we set  
 \begin{gather*}
  \theta_F := \trace^{k}_{F} \left( \omega - I^{k}_{\calW} \omega - \sum_{l=k}^{\dim F - 1} \cartan^{k-1} \xi^{l} \right)
  .
 \end{gather*}
 Let $m \in \{ k, \dots, n-1 \}$.
 Suppose that $\theta_f \in \underline{\mathring\calP}\Lambda^{k}(f)$ for each $f \in \calT^{m}$,
 which is certainly true if $m = k$.
 Then $\xi^{m}$ as in~\eqref{theorem:fluxreconstruction:higher} is well-defined.
 By assumptions on $\omega$ we find   
 \begin{align*}
  \cartan^{k} \theta_{f}
  &= 
  \cartan^{k} \trace^{k}_{f}
  \left( 
   \omega - I^{k}_{\calW} \omega - \sum_{l=k}^{m-1} \cartan^{k-1} \xi^{l}
  \right)
  \\&= 
  \trace^{k}_{f}
  \left( 
   \cartan^{k} \omega - \cartan^{k} I^{k}_{\calW} \omega - \cartan^{k} \sum_{l=k}^{m-1} \cartan^{k-1} \xi^{l}
  \right)
  =
  \trace^{k}_{f}
  \left( 
   \cartan^{k} \omega - I^{k+1}_{\calW} \cartan^{k} \omega
  \right)
  =
  0
  ,
 \end{align*}
 and conclude that $\cartan^{k-1} P_f^{k} \theta_f = \theta_f$. 
 In particular, 
 \begin{align}
  \label{math:beweis:temp}
  \trace^{k}_{f} \cartan^{k-1} \xi^{m}
  =
  \cartan^{k-1} P_{f}
  \theta_{f}
=
  \trace^{k}_{f}
  \left( 
   \omega - I^{k}_{\calW} \omega - \sum_{l=k}^{m-1} \cartan^{k-1} \xi^{l}
  \right)
  .
 \end{align}
 If $m < n$, then $\theta_F \in \underline{\mathring\calP}\Lambda^{k}(F)$ for each $F \in \calT^{m+1}$.
 The argument may be iterated until $m = n$. 
 In the latter case~\eqref{math:beweis:temp} provides~\eqref{theorem:fluxreconstruction:final}.
 
 Finally, if there exists $\xi \in \calP\Lambda^{k-1}(\calT,\calU)$ with $\cartan^{k-1} \xi = \omega$, then 
 \begin{gather*}
  I^{k}_{\calW} \omega = I^{k}_{\calW} \cartan^{k-1} \xi = \cartan^{k-1} I^{k-1}_{\calW} \xi.
 \end{gather*}
 and hence $\cartan^{k-1} \sfP_{\calW}^{k} I^{k}_{\calW} \xi = I^{k-1}_{\calW} \xi$,
 which shows~\eqref{theorem:fluxreconstruction:reallyfinal}.
 This completes the proof. 
\end{proof}

The theorem states that for every $\omega \in \calP\Lambda^{k}(\calT,\calU)$ with $\cartan^{k} \omega = 0$
there exists $\xi^{hi} \in \calP\Lambda^{k-1}(\calT,\calU)$ such that $\omega = I^{k}_{\calW} \omega + \cartan^{k} \xi^{hi}$.
If additionally $\omega$ is the exterior derivative of a member of $\calP\Lambda^{k-1}(\calT,\calU)$, 
then there exists $\xi^{lo} \in \calW\Lambda^{k-1}(\calT,\calU)$ with $\cartan^{k-1} \xi^{lo} = I^{k}_{\calW} \omega$,
so that $\xi := \xi^{lo} + \xi^{hi}$ is a solution of $\cartan^{k-1} \xi = \omega$. 

As a simple first application we address the dimension of the cohomology classes of the finite element de~Rham complex. 
This is a new proof of a known result~\cite{AFW2,StructPresDisc,licht2016complexes}. 
Conceptually, this shows that all cohomological information is completely encoded in the lowest-order component of the finite element de~Rham complex.  

\begin{lemma}
 The commuting interpolator $I^{k}_{\calW} : \calP\Lambda^{k}(\calT,\calU) \rightarrow \calW\Lambda^{k}(\calT,\calU)$
 induces isomorphisms on cohomology.
\end{lemma}

\begin{proof}
 Let $\omega \in \calW\Lambda^{k}(\calT,\calU)$ with $\cartan^{k}\omega = 0$.
 If $\omega \notin \cartan^{k-1} \calW\Lambda^{k-1}(\calT,\calU)$, 
 then $\omega \notin \cartan^{k-1} \calP\Lambda^{k-1}(\calT,\calU)$,
 since the canonical interpolant commutes with the exterior derivative. 
Hence $I^{k}_{\calW}$ induces a surjection on cohomology. 
Conversely, suppose that $\omega \in \calP\Lambda^{k}(\calT,\calU)$ and $\omega \notin \cartan^{k-1}\calP\Lambda^{k-1}(\calT,\calU)$.
 There exists $\xi \in \calP\Lambda^{k-1}(\calT,\calU)$
such that $\omega = \cartan^{k-1}\xi + I^{k}_{\calW} \omega$.
 Now $\omega \notin \cartan^{k-1}\calP\Lambda^{k-1}(\calT,\calU)$
 implies $I^{k}_{\calW} \omega \notin \cartan^{k-1}\calW\Lambda^{k-1}(\calT,\calU)$.
 Hence $I^{k}_{\calW}$ is injective on cohomology. 
This completes the proof.
\end{proof}

The partially localized flux reconstruction is relevant from a computational point of view, too. 
In order to compute a solution of $\cartan^{k-1} \xi = \omega$ for given $\omega \in \calP\Lambda^{k}(\calT,\calU)$
we treat this first-order equation as a least-squares problem. 
This means that we fix a Hilbert space structure on the finite element spaces 
and compute the action of the Moore-Penrose pseudoinverse of
$\cartan^{k-1} : \calP\Lambda^{k-1}(\calT,\calU) \rightarrow \calP\Lambda^{k}(\calT,\calU)$.
This is a standard topic of numerical linear algebra,
but the spectral properties of the operator $\cartan^{k-1} : \calP\Lambda^{k-1}(\calT,\calU) \rightarrow \calP\Lambda^{k}(\calT,\calU)$
for higher polynomial order can be disadvantageous. 
The condition number of the least-squares problem grows algebraically with the polynomial degree, 
which negatively affects the performance of the numerical methods. 
The complexity of the problem over higher-order spaces is comparable to the complexity of computing the flux variable in a mixed finite element method. 

But Theorem~\ref{theorem:fluxreconstruction} shows us how to avoid solving a global problem over a higher-order finite element space. 
We split the main problem into two independent subproblems:  
one subproblem involving Whitney forms and another subproblem involving higher-order contributions.
In the former subproblem we seek a flux reconstruction $\xi^{lo} \in \calW\Lambda^{k-1}(\calT,\calU)$
for $I^{k}_{\calP} \omega \in \calW\Lambda^{k}(\calT,\calU)$. 
Hence we still need to solve a global least-squares problem, but this time only 
for the operator $\cartan^{k-1} : \calW\Lambda^{k-1}(\calT,\calU) \rightarrow \calW\Lambda^{k}(\calT,\calU)$
over finite element spaces of lowest order.
In the second subproblem we calculate $\xi^{hi}$ by iterating over the dimension of the simplices in $\calT$ from lowest to highest; 
each step comprises of solving a block of mutually independent local subproblems.
In particular, each of these blocks is amenable to parallelization. 

In this sense the flux reconstruction is partially localized: 
the only remaining global operation involves a finite element space of merely lowest order instead of the full finite element space. 
A fully localized flux reconstruction is feasible when additional structure is provided; this will be crucial to our application in the next section. 

\begin{remark} \label{rem:fluxstability} 
  The $L^{2}$ stability of the global lowest-order problem depends only on the mesh quality and the domain,
  and the $L^{2}$ stability of the local problems depends only on the mesh quality and the polynomial order.
\end{remark}
 
\begin{example}
 \label{example:fluxreconstruction}
 We illustrate the partially localized flux reconstruction with two-dimensional finite element de~Rham complexes. 
 A completely analogous three-dimensional example is possible as well. 
 Assume that $\Omega \subset \bbR^{2}$ is a bounded simply-connected Lipschitz domain, 
 that $\calT$ is a triangulation of $\Omega$,
 and that $\calU \subset \calT$ triangulates $\partial\Omega$.
 
 We let $\calP_{r}(\calT,\calU)$ denote the functions over $\Omega$ that are piecewise polynomial of order $r$ with respect to $\calT$
 and satisfy Dirichlet boundary conditions. 
 We let $\calP_{r,\mathsf{DC}}(\calT)$ be the functions over $\Omega$ that are piecewise polynomial of order $r$ with respect to $\calT$.
 We let $\calP_{r}(\calT,\calU) \subseteq \calP_{r,\mathsf{DC}}(\calT)$ be the subspace whose members 
 have a square-integrable weak gradient and satisfy Dirichlet boundary conditions. 
 We let $\RT_{r}(\calT,\calU)$ and $\Ned_{r}(\calT,\calU)$ be the (divergence-conforming) Raviart-Thomas space over $\calT$ with normal boundary conditions along $\partial\Omega$.
At this point, we recall the divergence operator and the vector-valued curl operator;
 see also the next section for more details. 
 
 First, we perform the flux reconstruction for the divergence. 
 Let $f_h \in \calP_{r,\mathsf{DC}}(\calT)$ be a function over $\Omega$ that is piecewise in $\calP_r(\calT)$ and satisfies $\int_{\Omega} f_h = 0$.
 Then there exists $\xi_h \in {\RT}_{r}(\calT,\calU)$, generally not unique, with $\divergence \xi_h = f_h$. 
 To compute such a vector field, let $f_h' \in \calP_{r,\mathsf{DC}}(\calT)$ be the $L^{2}$ projection of $f_h$ onto the piecewise constant functions. 
 Then $\int_{\Omega} f_h' = \int_{\Omega} f_h = 0$, 
 and hence there exists $\xi'_h \in {\RT}_{0}(\calT,\calU)$ with vanishing normal components along $\partial\Omega$
 and $\divergence \xi'_h = f_h'$.
 Next, we let $f_h'' := f_h - f_h'$.
 For each $T \in \calT^{2}$ we have $\int_T f_h'' = 0$ by construction; 
 hence there exists $\xi_{T}'' \in \mathring{{\RT}}_{r}(T)$ with $\divergence \xi_{T}'' = f_{h|T}''$.
 We let $\xi_{h}'' := \sum_{T \in \calT^{2}} \xi_{T}''$ and $\xi_h := \xi_h' + \xi_h''$.
 Then $\xi_h  \in {\RT}_{r}(\calT,\calU)$ is the desired flux reconstruction. 
 
 Next, we explain the analogous flux reconstruction for the two-dimensional curl operator. 
 Suppose that $\theta_h \in {\RT}_{r}(\calT,\calU)$ is the curl of a member of $\calP_{r+1}(\calT,\calU)$.
 We let $\theta_h' \in {\RT}_{0}(\calT,\calU)$ be the canonical interpolation onto the lowest-order Raviart-Thomas space. 
 Since the canonical interpolation commutes with differential operators, 
 there exists $\sigma_h' \in \calP_1(\calT,\calU)$ with $\divergence \sigma_h' = \theta_h'$.
Let $\theta_h'' := \theta_h - \theta_h'$. 
 Note that $\theta_h''$ has a well-defined normal trace over the edges of $\calT$.  
 For every edge $E \in \calT^{1}$ we have $\int_E \vecn_E \cdot \theta_h'' = 0$.
 By taking a preimage with respect to differentiation of the function $\vecn_E \cdot \theta''$ over the edge $E$ 
 and extending the result onto the triangles that contain $E$,
 we conclude that there exists $\sigma_E'' \in \calP_{r+1}(\calT,\calU)$ supported on the two triangles adjacent to $E$
 with $\vecn_E \cdot ( \theta_h'' - \curl \sigma_E'' ) = 0$.
 We let $\sigma_h'' = \sum_{E \in \calT^{1}} \sigma_E''$ and let $\theta_h''' := \theta'' - \curl \sigma_h''$.
 By construction, we can write $\theta_h''' = \sum_{T \in \calT^{2}} \theta_{T}'''$ 
 where for each $T \in \calT^{2}$ we have $\theta_{T}''' \in \mathring{{\RT}}_{r}(T)$ and $\divergence \theta_T''' = 0$.
 For each triangle $T \in \calT^{2}$ there exists $\sigma_T''' \in \mathring \calP_{r+1}(T)$ with $\curl \sigma_T''' = \theta_T'''$.
 We set $\sigma_h''' = \sum_{T \in \calT^{2}} \sigma_T'''$.
 Eventually, we let $\sigma_h := \sigma_h' + \sigma_h'' + \sigma_h'''$
 and observe $\curl \sigma_h = \theta_h$.
\end{example}

\section{Applications in A Posteriori Error Estimation} \label{sec:application}

We apply the partially localized flux reconstruction 
to obtain a fully localized flux reconstruction for equilibrated a~posteriori estimation. 
We illustrate the idea with the $\curl$-$\curl$ equation over a two-dimensional domain. 
Three-dimensional domains are handled analogously. 
We generalize the equilibrated a~posteriori error estimator for N\'ed\'elec elements of~\cite{BrSchoMax}
from the case of lowest-order to case of higher and possibly non-uniform polynomial order. 
The following discussion focuses on the theoretical framework. 
Extensive computational studies will be subject of subsequent work. 
\sectionintrolinebreak

This section extends the discussion in Example~\ref{example:fluxreconstruction}. 
Let $\Omega \subseteq \bbR^{2}$ be a bounded Lipschitz domain. We let $L^{2}(\Omega)$ and $\bfL^{2}(\Omega) := L^{2}(\Omega)^{2}$ denote the Hilbert spaces 
of square-integrable functions and vector fields, respectively, over $\Omega$. 
The corresponding scalar products and norms are written $\langle\cdot,\cdot\rangle_{L^{2}}$ and $\|\cdot\|_{L^{2}}$, respectively. 
We let $H^{1}(\Omega)$ be the first-order Sobolev space 
and let $\bfH(\divergence,\Omega)$ be the space of square-integrable vector fields with divergence in $L^{2}(\Omega)$.
These are Hilbert spaces endowed with the respective graph scalar products of the gradient and the divergence, 
\begin{gather*}
 \grad : H^{1}(\Omega) \rightarrow \bfL^{2}(\Omega), \quad \omega \mapsto ( \partial_x \omega, \partial_y \omega ),
 \\
 \divergence : \bfH(\divergence,\Omega) \rightarrow L^{2}(\Omega), \quad (u,v) \mapsto \partial_x u + \partial_y v.
\end{gather*}
Consider the isometry $J : \bfL^{2}(\Omega) \rightarrow \bfL^{2}(\Omega)$ 
which rotates each vector field by a right angle counterclockwise, i.e. $J(u,v) = (-v,u)$ for $(u,v) \in \bfL^{2}(\Omega)$.
We introduce  
\begin{gather*}
 \bfH(\curl,\Omega) := J^{\inv}\bfH(\divergence,\Omega)
\end{gather*}
and introduce the differential operators 
\begin{gather*}
 \curl : \bfH(\curl,\Omega) \rightarrow L^{2}(\Omega), \quad \nu \mapsto \divergence J \nu,
 \\
 \curl : H^{1}(\Omega) \rightarrow \bfL^{2}(\Omega), \quad \tau \mapsto J \grad \tau.
\end{gather*}
We have Hilbert spaces $H^{1}(\Omega)$ and $\bfH(\curl,\Omega)$ with the respective graph scalar products. 
To formalize boundary conditions, we let $H^{1}_{0}(\Omega)$, $\bfH_{0}(\divergence,\Omega)$ and $\bfH_{0}(\curl,\Omega)$ 
denote the closure of the compactly supported smooth scalar or vector fields
over $\Omega$ in $H^{1}(\Omega)$, $\bfH(\divergence,\Omega)$ and $\bfH(\curl,\Omega)$, respectively.
It is easy to see that we have well-defined differential complexes 
\begin{gather}
\label{math:hilbertcomplex:primal}
 \begin{CD}
  0 \to \bbR @>>> H^{1}(\Omega) @>{\grad}>> \bfH(\curl,\Omega) @>\curl>> L^{2}(\Omega) \to 0,
 \end{CD}
\\
 \label{math:hilbertcomplex:dual}
 \begin{CD}
  0 \gets \bbR @<{\int}<< L^{2}(\Omega) @<{\divergence}<< \bfH_{0}(\divergence,\Omega) @<\curl<< H^{1}_{0}(\Omega) \gets 0.
 \end{CD}
\end{gather}
Here, the differential operators have closed range 
and both differential complexes are mutually $L^{2}$ adjoint
(as Hilbert complexes in the sense of~\cite{bruening1992hilbert}). 
If moreover the domain is simply-connected, then the differential complexes 
\eqref{math:hilbertcomplex:primal} and~\eqref{math:hilbertcomplex:dual} are exact~\cite{AFW1}.
We also recall the integration by parts formulas 
\begin{gather}
 \label{math:ibp:curlcurl}
 \langle \curl \nu, \tau \rangle_{L^{2}} = \langle \nu, \curl \tau \rangle_{L^{2}},
 \quad 
 \nu \in \bfH(\curl,\Omega), \quad \tau \in H^{1}_{0}(\Omega),
 \\
 \label{math:ibp:graddiv}
 \langle \grad v, \nu \rangle_{L^{2}} = - \langle v, \divergence \nu \rangle_{L^{2}},
 \quad 
 v \in H^{1}(\Omega), \quad \nu \in \bfH_{0}(\divergence,\Omega).
\end{gather}
The $\curl$-$\curl$ problem is to find a vector field $\upsilon$ that satisfies $\curl \curl \upsilon = \theta$ for a given vector field $\theta$.
We consider the weak formulation over Sobolev spaces of vector fields, 
where we assume that $\theta \in \bfL^{2}(\Omega)$ and search for $\upsilon \in \bfH(\curl,\Omega)$ with 
\begin{gather} \label{math:curlcurlbeispiel}
 \langle \curl \upsilon, \curl \nu \rangle_{L^{2}}
 =
 \langle \theta, \nu \rangle_{L^{2}},
 \quad 
 \nu \in \bfH(\curl,\Omega).
\end{gather}
Solutions of~\eqref{math:curlcurlbeispiel} are generally not unique, 
because the $\curl$ operator has a non-trivial kernel.
To ensure uniqueness, one may impose $\upsilon$ to be orthogonal to the gradients of functions in $H^{1}(\Omega)$,
which enforces $\upsilon \in \bfH_{0}(\divergence,\Omega)$ with $\divergence \upsilon = 0$.
Singling out a specific solution, however, is only peripheral to our discussion. 

If we additionally assume that $\theta \in \bfH_{0}(\divergence,\Omega)$ with $\divergence \theta = 0$,
so that $\theta$ is the $\curl$ of a scalar function in $H^{1}_{0}(\Omega)$,
then basic results in operator theory already imply that $\curl \upsilon \in H^{1}_{0}(\Omega)$ with $\curl \curl \upsilon = \theta$,
that is, the weak solution $\upsilon$ of~\eqref{math:curlcurlbeispiel} even is a strong solution. 
\\

In order to address a~posteriori error estimation 
we fix a solution $\upsilon \in \bfH(\curl,\Omega)$ and let $\upsilon_h \in \bfH(\curl,\Omega)$ be arbitrary. 
We let $\sigma \in H^{1}_{0}(\Omega)$ with $\curl \sigma = \theta$.
Now, 
\begin{align*}
 &\| \sigma - \curl \upsilon_h \|^{2}_{L^{2}}
\\&\quad=
 \| \sigma - \curl \upsilon \|^{2}_{L^{2}}
 +
 \| \curl \upsilon - \curl \upsilon_h \|^{2}_{L^{2}}
- 
 2 \left\langle \sigma - \curl \upsilon, \curl \upsilon - \curl \upsilon_h \right\rangle_{L^{2}}
 . 
\end{align*}
Using~\eqref{math:ibp:curlcurl} and $\curl \sigma = \theta = \curl \upsilon$,
we note   
\begin{align*}
 &
 \langle \sigma - \curl \upsilon, \curl \upsilon - \curl \upsilon_h \rangle_{L^{2}}
 \\&\quad
 =
 \langle \curl ( \sigma - \curl \upsilon ), \upsilon - \upsilon_h \rangle_{L^{2}}
 =
 \langle \theta - \theta, \upsilon - \upsilon_h \rangle_{L^{2}}
 =
 0
 .
\end{align*}
Thus,
\begin{align} \label{math:pragersynge}
 \| \sigma - \curl \upsilon_h \|^{2}_{L^{2}}
 =
 \| \sigma - \curl \upsilon \|^{2}_{L^{2}}
 +
 \| \curl \upsilon - \curl \upsilon_h \|^{2}_{L^{2}}
 .  
\end{align}
Equation~\eqref{math:pragersynge} is a generalized Prager-Synge identity~\cite{BrSchoMax}.

We utilize this result as follows.
Let $\upsilon \in \bfH(\curl,\Omega)$ with $\curl\upsilon \in H^{1}_{0}(\Omega)$ be a strong solution of~\eqref{math:curlcurlbeispiel}.
Given any exact solution $\sigma \in H^{1}_{0}(\Omega)$ of $\curl \sigma = \theta$ 
and any $\upsilon_h \in \bfH(\curl,\Omega)$,
it now follows that 
\begin{gather} \label{math:pragersynge:application}
 \| \sigma - \curl \upsilon_h \|_{L^{2}}
 \geq 
 \| \curl \upsilon - \curl \upsilon_h \|_{L^{2}}
 .
\end{gather}
The left-hand side of~\eqref{math:pragersynge:application} is given in terms of known objects 
and dominates the right-hand side of~\eqref{math:pragersynge:application},
which depends on the generally unknown true solution $\upsilon$.
If $\upsilon_h$ is seen as an approximation of $\upsilon$, 
then we see~\eqref{math:pragersynge:application} as an error estimate for the derivatives.

In a typical application, $\upsilon_h$ is the Galerkin solution of a finite element method.
If we have an exact solution $\sigma \in H^{1}_{0}(\Omega)$ of $\curl \sigma = \theta$, 
then~\eqref{math:pragersynge:application} gives an upper bound on one component of the error in the $\bfH(\curl,\Omega)$ norm. 
The exact flux $\curl \upsilon$ is generally unknown and hence not a candidate for $\sigma$,
but we can find a candidate via flux reconstruction. 
\\ 
 
As a technical preparation, we consider finite element de~Rham complexes over the domain $\Omega$.
Let $\calT$ be a simplicial complex triangulating $\Omega$
and let $\calU$ denote the subcomplex of $\calT$ triangulating $\partial\Omega$.
The latter is merely a finite set of line segments in this case. 
We focus on higher-order finite element spaces of uniform order;
the generalization to spaces of non-uniform polynomial order is straight forward. 
Let $r \in \bbN_0$ and recall the N\'ed\'elec space $\Ned_{r}(\calT)$ of polynomial order $r$ with respect to $\calT$. 
Consider the finite element de~Rham complexes  
\begin{gather*}
 \begin{CD}
  0 \to \bbR @>>> \calP_{r+1}(\calT) @>\grad>> \Ned_{r}(\calT) @>\curl>> \calP_{r,\mathrm{DC}}(\calT) \to 0  
 \end{CD}
\end{gather*}
and 
\begin{gather*}
 \begin{CD}
  0 \gets \bbR @<\int<< \calP_{r,\mathrm{DC}}(\calT) @<\divergence<< \RT_{r}(\calT,\calU) @<\curl<< \calP_{r+1}(\calT,\calU) \gets 0
  .
 \end{CD}
\end{gather*}
The first is a finite-di\-men\-sio\-nal subcomplex of~\eqref{math:hilbertcomplex:primal}
and the second is a finite-di\-men\-sio\-nal subcomplex of~\eqref{math:hilbertcomplex:dual}.

Let $\theta \in \bfH_{0}(\divergence,\Omega)$ be as before 
but assume additionally that $\theta \in \RT_{r}(\calT,\calU)$.
Then there exists a member of $\calP_{r+1}(\calT,\calU)$ whose curl equals $\theta$.
In order to utilize the Prager-Synge identity and estimate~\eqref{math:pragersynge:application},
it remains to algorithmically construct a generalized inverse for the operator 
\begin{gather}
 \label{math:finiteelementcurl}
 \curl : \calP_{r+1}(\calT,\calU) \rightarrow \RT_{r}(\calT,\calU).
\end{gather}
This is achieved via partially localized flux reconstruction. 
Using the construction in Example~\ref{example:fluxreconstruction}, 
we decompose 
\begin{gather*}
 \theta = \theta_0 + \curl \xi_r, 
\end{gather*}
where $\theta_0 \in \RT_{0}(\calT,\calU)$ is the canonical interpolation of $\theta$ 
onto the lowest-order Raviart-Thomas space with homogeneous normal boundary conditions 
and where $\xi_r \in \calP_{r+1}(\calT,\calU)$ is computed through a number of local problems over simplices 
whose computation is parallelizable. 
This reduces the least-squares problem to the special case $r=0$.
\\

The partially locally flux reconstruction can be extended to a \emph{fully localized} flux reconstruction 
if additional information about $\theta$ is given. 
Specifically, assume that $\upsilon_h \in \Ned_{r}(\calT,\calU)$ satisfies the Galerkin condition
\begin{gather}
 \label{math:curlcurlbeispiel:discrete}
 \langle \curl \upsilon_{h}, \curl \nu_{h} \rangle_{L^{2}}
 =
 \langle \theta, \nu_{h} \rangle_{L^{2}},
 \quad 
 \nu_{h} \in \Ned_{r}(\calT,\calU).
\end{gather}
Local computations provide $\theta_0 \in \RT_{0}(\calT,\calU)$ and $\xi_r \in \calP_{r+1}(\calT,\calU)$
such that $\theta = \theta_0 + \curl \xi_r$.
We observe that 
\begin{align*}
    \langle \curl \xi_r, \nu_h \rangle_{L^{2}} 
    =
    \langle \xi_r, \curl \nu_h \rangle_{L^{2}}
    ,
    \quad 
    \nu_h \in \bfH(\curl,\Omega) 
    .
\end{align*}
Let $\gamma_h \in \calP_{0,\mathrm{DC}}(\calT)$ be the $L^{2}$ orthogonal projection of $\xi_r - \curl \upsilon_h$ onto $\calP_{0,\mathrm{DC}}(\calT)$.
We note $\gamma_h$ can be computed for each simplex independently. 
Thus 
\begin{gather*}
 \langle \gamma_{h}, \tau_{h} \rangle_{L^{2}}
 =
 \left\langle \xi_r - \curl \upsilon_h, \tau_{h} \right\rangle_{L^{2}}
 ,
 \quad 
 \tau_{h} \in \calP_{0,\mathrm{DC}}(\calT)
 .
\end{gather*}
We then find for $\nu_h \in \Ned_{0}(\calT,\calU)$ that 
\begin{align*}
 0
 &=
 \langle \theta, \nu_h \rangle_{L^{2}}
 -
 \langle \curl \upsilon_h, \curl \nu_h \rangle_{L^{2}} 
 \\&=
 \langle \theta_0 + \curl \xi_r, \nu_h \rangle_{L^{2}} 
 -
 \langle \curl \upsilon_h, \curl \nu_h \rangle_{L^{2}} 
 \\&=
 \langle \theta_0, \nu_h \rangle_{L^{2}} 
 +
 \langle \xi_r - \curl \upsilon_h, \curl \nu_h \rangle_{L^{2}} 
 =
 \langle \theta_{0}, \nu_h \rangle_{L^{2}} 
 -
 \langle \gamma_h, \curl \nu_h \rangle_{L^{2}} 
\end{align*}
because of the Galerkin orthogonality~\eqref{math:curlcurlbeispiel:discrete} and $\curl \nu_h \in \calP_{0,\mathrm{DC}}(\calT)$.
Moreover, $\divergence \theta_0 = 0$ since the canonical interpolator commutes with the differential operators.
The next crucial step uses the construction in~\cite{BrSchoMax}
to find $\varrho_h \in \calP_{1,\mathrm{DC}}(\calT)$ with 
\begin{gather}
 \label{math:localizedfluxreconstruction}
 \langle \varrho_h, \curl \nu \rangle_{L^{2}}
 =
 \langle \theta_0, \nu \rangle_{L^{2}} + \langle \gamma_h, \curl \nu \rangle_{L^{2}},
 \quad 
 \nu \in \bfH(\curl,\Omega),
\end{gather}
where $\varrho_h$ is computed by solving localized problems over element patches around vertices. 
Hence, for all $\nu \in \bfH(\curl,\Omega)$ we observe 
\begin{align*}
 &
 \langle \theta, \nu \rangle_{L^{2}}
 -
 \langle \curl \upsilon_h, \curl \nu \rangle_{L^{2}} 
 \\&\quad=
 \langle \theta_0, \nu \rangle_{L^{2}} 
 +
 \langle \gamma_h, \curl \nu \rangle_{L^{2}} 
 +
 \langle \xi_r, \curl \nu \rangle_{L^{2}} 
 -
 \langle \gamma_h, \curl \nu \rangle_{L^{2}} 
 -  
 \langle \curl\upsilon_h, \curl \nu \rangle_{L^{2}} 
 \\&\quad=
 \langle \varrho_h, \curl \nu \rangle_{L^{2}} 
 +
 \langle \xi_r - \gamma_h - \curl\upsilon_h, \curl \nu \rangle_{L^{2}} 
 \\&\quad=
 \langle \varrho_h + \xi_r - \gamma_h - \curl\upsilon_h, \curl \nu \rangle_{L^{2}} 
 .
\end{align*}
We set \begin{gather*}
 \sigma_{h} := \varrho_h + \xi_r - \gamma_h.
\end{gather*}
The above results show that 
\begin{align*}
 \langle \theta, \nu \rangle_{L^{2}}
 =
 \langle \sigma_h, \curl \nu \rangle_{L^{2}}, \quad \nu \in \bfH(\curl,\Omega), 
\end{align*}
which implies that $\sigma_h \in H^{1}_{0}(\Omega)$.
In particular, $\sigma_h \in \calP_{r+1}(\calT,\calU)$ with 
\begin{align*}
 \theta
 &=
 \curl \sigma_h.
\end{align*}
The function $\sigma_h$ has been constructed only by local computations.
This completes the flux reconstruction and enables the a~posteriori error estimate~\eqref{math:pragersynge:application}.

\begin{remark} \label{rem:fluxalgorithm}  
    We have described the partially localized flux reconstruction recursively over the dimensions.
    At each level, a parallelizable block of local mutually independent computations is performed. 
    But we also see that $\xi_r$ (and $\gamma_h$) over each volume depends only on local information over that volume. 
    Hence, one can rearrange the calculations such that, 
    at the cost of redundant computations,
    we only process one parallelizable block of mutually independent local problems associated to full-dimensional simplices. 
    At the cost of even more redundant computations,
    $\sigma_h$ can be constructed with a single parallelizable block of problems associated to patches.     
    Via Remark~\ref{rem:fluxstability} we furthermore see that the stability of the construction of $\sigma_h$
    depends only on the mesh quality, the domain, and the polynomial order of the finite element spaces.
\end{remark}

\begin{remark} We briefly comment on the reliability and efficiency of the error estimator.
    The \emph{reliability} of the estimator is already expressed by~\eqref{math:pragersynge}. 
    The \emph{efficiency} of the estimator requires a converse inequality of the form 
    \begin{align*}
        \| \curl \upsilon_{h} - \curl \sigma_{h} \|_{L^{2}} 
        \leq 
        C 
        \| \curl \upsilon - \curl \upsilon_{h} \|_{L^{2}}
    \end{align*}
    for some constant $C > 0$. 
    The line of argument is analogous to the lowest-order case (see Theorem~4 and Theorem~13 in~\cite{BrSchoMax}):
    one uses the stability of the flux reconstruction, which follows from scaling arguments, 
    and comparison to the residual-based error estimator, which is known to be efficient. 
    With the assumptions made in this section, 
    in particular that $\theta \in \bfH_{0}(\divergence,\Omega) \cap \RT_{r}(\calT,\calU)$,
    one thus sees the existence of such an efficiency constant $C > 0$
    that depends only on $\Omega$, $r$, and the shape regularity of the mesh.

We have assumed that $\theta \in \RT_{r}(\calT,\calU)$ is divergence-free and lies in a finite element space. 
    More generally, 
    if we reconstruct the flux using a finite element interpolation
    $\theta \in \RT_{r}(\calT,\calU)$ of some divergence-free right-hand side
    $\tilde\theta \in \bfH_{0}(\divergence,\Omega)$,
    then additional error terms, known as \emph{data oscillation}, appear,
    just as in the scalar-valued case (see~\cite{braess2009equilibrated}). 
Going further, 
    if $\theta$ is not divergence-free
    or is not orthogonal to any finite element harmonic vector field in $\RT_{r}(\calT,\calU)$,
    then another type of additional error term appears. 
    In that case the weak formulation~\eqref{math:curlcurlbeispiel}
    admits a solution only in the sense of least squares and the Galerkin condition~\eqref{math:curlcurlbeispiel:discrete}
    will not be satisfied for all $\nu_{h} \in \Ned_{r}(\calT,\calU)$.
    The practical consequence is that the localized linear problems in the construction of Braess and Sch\"oberl
    will generally not admit exact solutions, and the same holds for our generalization to the higher-order case.
\end{remark}

Our techniques apply similarly to higher-order flux reconstruction for edge elements in dimension three.
    Again, the lowest-order case is treated in~\cite{BrSchoMax}, and the extension to the higher order case is the same almost word by word.

\subsection*{Acknowledgments}
S\"oren Bartels is thanked for bringing the original publication by Dietrich Braess and Joachim Sch\"oberl to the author's attention.
This research was supported by the European Research Council through the FP7-IDEAS-ERC Starting Grant scheme, project 278011 STUCCOFIELDS.
This work is based on the author's PhD thesis.


\begin{thebibliography}{10}

\bibitem{Ainsworth20016709}
Mark Ainsworth and Joe Coyle.
\newblock Hierarchic $hp$-edge element families for {M}axwell's equations on
  hybrid quadrilateral/triangular meshes.
\newblock {\em Computer Methods in Applied Mechanics and Engineering},
  190(49-50):6709--6733, 2001.

\bibitem{ainsworth2011posteriori}
Mark Ainsworth and J~Tinsley Oden.
\newblock {\em A {P}osteriori {E}rror {E}stimation in {F}inite {E}lement
  {A}nalysis}, volume~37 of {\em Pure and Applied Mathematics: A Wiley Series
  of Texts, Monographs, and Tracts}.
\newblock John Wiley \& Sons, Hoboken, NY, 2011.

\bibitem{AFW1}
Douglas~N. Arnold, Richard~S. Falk, and Ragnar Winther.
\newblock Finite element exterior calculus, homological techniques, and
  applications.
\newblock {\em Acta Numerica}, 15:1--155, May 2006.

\bibitem{afwgeodecomp}
Douglas~N. Arnold, Richard~S. Falk, and Ragnar Winther.
\newblock Geometric decompositions and local bases for spaces of finite element
  differential forms.
\newblock {\em Computer Methods in Applied Mechanics and Engineering},
  198(21-26):1660--1672, May 2009.

\bibitem{AFW2}
Douglas~N. Arnold, Richard~S. Falk, and Ragnar Winther.
\newblock Finite element exterior calculus: from {Hodge} theory to numerical
  stability.
\newblock {\em Bulletin of the American Mathematical Society}, 47(2):281--354,
  2010.

\bibitem{becker2015robust}
Roland Becker, Daniela Capatina, and Robert Luce.
\newblock Robust local flux reconstruction for various finite element methods.
\newblock In {\em Numerical Mathematics and Advanced Applications-ENUMATH
  2013}, pages 65--73. Springer, 2015.

\bibitem{becker2016local}
Roland Becker, Daniela Capatina, and Robert Luce.
\newblock Local flux reconstructions for standard finite element methods on
  triangular meshes.
\newblock {\em SIAM Journal on Numerical Analysis}, 54(4):2684--2706, 2016.

\bibitem{beuchler2007sparse}
Sven Beuchler and Veronika Pillwein.
\newblock Sparse shape functions for tetrahedral $p$-{FEM} using integrated
  {J}acobi polynomials.
\newblock {\em Computing}, 80(4):345--375, 2007.

\bibitem{bossavit1988mixed}
Alain Bossavit.
\newblock Mixed finite elements and the complex of {W}hitney forms.
\newblock {\em The mathematics of finite elements and applications VI}, pages
  137--144, 1988.

\bibitem{Braess2007}
Dietrich Braess.
\newblock {\em Finite {E}lements - {T}heory, {F}ast {S}olvers, and
  {A}pplications in {E}lasticity {T}heory}.
\newblock Cambridge University Press, Cambridge, 3rd edition, 2007.

\bibitem{braess2009equilibrated}
Dietrich Braess, Veronika Pillwein, and Joachim Sch{\"o}berl.
\newblock Equilibrated residual error estimates are $p$-robust.
\newblock {\em Computer Methods in Applied Mechanics and Engineering},
  198(13-14):1189--1197, 2009.

\bibitem{BrSchoMax}
Dietrich Braess and Joachim Sch{\"o}berl.
\newblock Equilibrated residual error estimator for edge elements.
\newblock {\em Mathematics of Computation}, 77(262):651--672, 2008.

\bibitem{bruening1992hilbert}
Jochen Br\"uning and Matthias Lesch.
\newblock {H}ilbert complexes.
\newblock {\em Journal of Functional Analysis}, 108(1):88--132, 1992.

\bibitem{carstensen2010estimator}
Carsten Carstensen and Christian Merdon.
\newblock Estimator competition for {P}oisson problems.
\newblock {\em Journal of Computational Mathematics}, 3:309--330, 2010.

\bibitem{christiansen2008construction}
Snorre~H Christiansen.
\newblock A construction of spaces of compatible differential forms on cellular
  complexes.
\newblock {\em Mathematical Models and Methods in Applied Sciences},
  18(05):739--757, 2008.

\bibitem{christiansen2020poincare}
Snorre~H Christiansen and Martin~W Licht.
\newblock Poincar{\'e}--friedrichs inequalities of complexes of discrete
  distributional differential forms.
\newblock {\em BIT Numerical Mathematics}, 60(2):345--371, 2020.

\bibitem{StructPresDisc}
Snorre~Harald Christiansen, Hans~Z. Munthe-Kaas, and Brynjulf Owren.
\newblock Topics in structure-preserving discretization.
\newblock {\em Acta Numerica}, 20:1--119, May 2011.

\bibitem{Demkowicz2001}
Leszek Demkowicz.
\newblock {\em Edge Finite Elements of Variable Order for {M}axwell's
  Equations}, pages 15--34.
\newblock Springer Berlin Heidelberg, Berlin, Heidelberg, 2001.

\bibitem{demkowicz2006computing}
Leszek Demkowicz.
\newblock {\em Computing with $hp$-adaptive {F}inite {E}lements: {V}olume 1:
  {O}ne and {T}wo {D}imensional {E}lliptic and {M}axwell problems}.
\newblock CRC Press, New York, 2006.

\bibitem{deRhamHPFEM}
Leszek Demkowicz, Peter Monk, Leon Vardapetyan, and Waldemar Rachowicz.
\newblock De {R}ham diagram for $hp$ finite element spaces.
\newblock {\em Computers \& Mathematics with Applications}, 39:29--28, 2000.

\bibitem{ducret2009lq}
Stephen Ducret.
\newblock {\em $L_{p,q}$-cohomology of {R}iemannian manifolds and simplicial
  complexes of bounded geometry}.
\newblock PhD thesis, EPFL, 2009.

\bibitem{ern2010guaranteed}
Alexandre Ern, Annette~F Stephansen, and Martin Vohral{\'\i}k.
\newblock Guaranteed and robust discontinuous galerkin a posteriori error
  estimates for convection--diffusion--reaction problems.
\newblock {\em Journal of computational and applied mathematics},
  234(1):114--130, 2010.

\bibitem{ern2010posteriori}
Alexandre Ern and Martin Vohral{\'\i}k.
\newblock A posteriori error estimation based on potential and flux
  reconstruction for the heat equation.
\newblock {\em SIAM Journal on Numerical Analysis}, 48(1):198--223, 2010.

\bibitem{FriAgri}
Thomas Friedrich and Ilka Agricola.
\newblock {\em {Global Analysis: Differential Forms in Analysis, Geometry, and
  Physics}}, volume~52.
\newblock American Mathematical Society, Providence, RI, 2002.

\bibitem{hiptmair2001higher}
Ralf Hiptmair.
\newblock Higher order {W}hitney forms.
\newblock {\em Progress in Electromagnetics Research}, 32:271--299, 2001.

\bibitem{hiptmair2002finite}
Ralf Hiptmair.
\newblock Finite elements in computational electromagnetism.
\newblock {\em Acta Numerica}, 11(1):237--339, 2002.

\bibitem{kirby2014low}
Robert~C. Kirby.
\newblock Low-complexity finite element algorithms for the de {R}ham complex on
  simplices.
\newblock {\em SIAM Journal on Scientific Computing}, 36(2):A846--A868, 2014.

\bibitem{kirby2012fast}
Robert~C. Kirby and Kieu~Tri Thinh.
\newblock Fast simplicial quadrature-based finite element operators using
  {B}ernstein polynomials.
\newblock {\em Numerische Mathematik}, 121(2):261--279, 2012.

\bibitem{licht2016complexes}
Martin~Werner Licht.
\newblock Complexes of discrete distributional differential forms and their
  homology theory.
\newblock {\em Foundations of Computational Mathematics}, pages 1--38, 2016.

\bibitem{melenk2002condition}
Jens~Markus Melenk.
\newblock On condition numbers in $hp$-{FEM} with {G}auss--{L}obatto-based
  shape functions.
\newblock {\em Journal of Computational and Applied Mathematics},
  139(1):21--48, 2002.

\bibitem{MONK1994117}
Peter Monk.
\newblock On the $p$- and $hp$-extension of {N}\'ed\'elec's $\curl$-conforming
  elements.
\newblock {\em Journal of Computational and Applied Mathematics}, 53(1):117 --
  137, 1994.

\bibitem{Monk2003}
Peter Monk.
\newblock {\em {F}inite {E}lement {M}ethods for {M}axwell's {E}quations}.
\newblock Clarendon Press, Oxford, 2003.

\bibitem{rachowicz2006fully}
Waldemar Rachowicz, David Pardo, and Leszek Demkowicz.
\newblock Fully automatic hp-adaptivity in three dimensions.
\newblock {\em Computer methods in applied mechanics and engineering},
  195(37):4816--4842, 2006.

\bibitem{repin2009two}
Sergey Repin and Rolf Stenberg.
\newblock Two-sided a posteriori estimates for a generalized stokes problem.
\newblock {\em Journal of Mathematical Sciences}, 159:541--558, 2009.

\bibitem{repin2008posteriori}
Sergey~I. Repin.
\newblock {\em A {P}osteriori {E}stimates for {P}artial {D}ifferential
  {E}quations}, volume~4 of {\em Radon Series on Computational and Applied
  Mathematics}.
\newblock Walter de Gruyter, Berlin, 2008.

\bibitem{schoberl2005high}
Joachim Sch{\"o}berl and Sabine Zaglmayr.
\newblock High order {N}{\'e}d{\'e}lec elements with local complete sequence
  properties.
\newblock {\em COMPEL-The international journal for computation and mathematics
  in electrical and electronic engineering}, 24(2):374--384, 2005.

\bibitem{S98_283}
Christoph Schwab.
\newblock {\em $p$-- and $hp$--{F}inite {E}lement {M}ethods - {T}heory and
  {A}pplications in {S}olid and {F}luid {M}echanics}.
\newblock Oxford Univ. Press, 1998.

\bibitem{whitney2012geometric}
Hassler Whitney.
\newblock {\em Geometric {I}ntegration {T}heory}.
\newblock Princeton University Press, 1957.

\bibitem{zaglmayr2006high}
Sabine Zaglmayr.
\newblock {\em High order finite element methods for electromagnetic field
  computation}.
\newblock PhD thesis, Universit{\"a}t Linz, Dissertation, 2006.

\end{thebibliography}
\end{document}